\documentclass[a4paper,reqno]{amsart}
\usepackage[english]{babel}
\usepackage[latin1]{inputenc}
\usepackage{times}
\usepackage{amsmath, amssymb, amstext, amsthm, graphicx, enumerate, caption, accents}
\usepackage{subcaption}
\makeindex
\usepackage{color} 
\definecolor{MyBlue}{rgb}{0,0.2,0.6}
\usepackage[linktocpage=true,colorlinks=true,linkcolor=blue,citecolor=blue, pdfstartview={FitH}]{hyperref}
\addto\extrasenglish{}%
\addto\extrasenglish{}

\newcommand{\R}{\mathbb{R}}
\newcommand{\C}{\mathbb{C}}
\newcommand{\Z}{\mathbb{Z}}
\newcommand{\mc}{{\Small{MC}}}

\newcommand{\mcH}{{\Small{MC}}$H$}

\newcommand{\mcone}{{\Small{MC}}$1$}
\newcommand{\Lcal}{\mathcal{L}}

\newcommand{\Mcal}{\mathcal{M}}

\newcommand{\N}{\mathbb{N}}
\newcommand{\bone}{\boldsymbol{1}} 
\newcommand{\ii}{\boldsymbol{i}}  
\newcommand{\jj}{\boldsymbol{j}}  
\newcommand{\kk}{\boldsymbol{k}}  
\renewcommand{\H}{\mathbb{H}}

\DeclareMathOperator{\id}{id}

\newcommand{\Sigmav}{\Sigma_{\operatorname{v}}}
\newcommand{\Sigmanv}{\Sigma_{\operatorname{nv}}}

\theoremstyle{plain}

\newtheorem*{Conjnn}{Conjecture}
\newtheorem{Theorem}{Theorem}[section]
\newtheorem{Cor}{Corollary}
\newtheorem{Lem}{Lemma}
\newtheorem{Conj}{Conjecture}
\newtheorem*{Notation}{Notation}
\newtheorem*{Question}{Question}

\newtheorem{Obs}{Observation}
\newtheorem{Prop}{Proposition}
\theoremstyle{definition}
\newtheorem*{Def}{Definition}
\newtheorem{Ex}{Example}

\newtheorem{Rem}{Remark}
\newtheoremstyle{exercise}
  {8pt}                   
  {2pt}                   
  {\normalfont}           
  {}                      
  {\normalfont\bfseries}  
  {}                      
  {\newline}              
  {}
\theoremstyle{exercise}

\makeatletter
\renewcommand*{\c@Lem}{\c@Theorem}
\renewcommand*{\p@Lem}{\p@Theorem}

\renewcommand*{\c@Cor}{\c@Theorem}
\renewcommand*{\p@Cor}{\p@Theorem}

\renewcommand*{\c@Prop}{\c@Theorem}
\renewcommand*{\p@Prop}{\p@Theorem}

\renewcommand*{\c@Ex}{\c@Theorem}
\renewcommand*{\p@Ex}{\p@Theorem}

\renewcommand*{\c@Exer}{\c@Theorem}
\renewcommand*{\p@Exer}{\p@Theorem}

\renewcommand*{\c@Rem}{\c@Theorem}
\renewcommand*{\p@Rem}{\p@Theorem}

\renewcommand*{\c@Conj}{\c@Theorem}
\renewcommand*{\p@Conj}{\p@Theorem}

\def\widebar{\accentset{{\cc@style\underline{\mskip10mu}}}}
\def\Widebar{\accentset{{\cc@style\underline{\mskip8mu}}}}
\makeatother

\begin{document}
\title{On the existence problem for tilted unduloids in $\mathbb{H}^2\times\R$}
\date{\today}
\author{Miroslav Vr\v{z}ina}
\address{Technische Universit\"at Darmstadt, Fachbereich Mathematik (AG 3 ``Geometrie und Approximation''),
	Schlossgartenstr.~7, 64289 Darmstadt, Germany}
\email{vrzina@mathematik.tu-darmstadt.de}
\subjclass[2000]{Primary 53A10; Secondary 53C22, 53C30}

\keywords{Differential geometry, invariant surfaces, minimal, constant mean curvature, homogeneous 3-manifolds, geodesics}
\begin{abstract}
	We study the existence problem for tilted unduloids in $\mathbb{H}^2\times\R$. These are singly periodic annuli with constant mean curvature $H>1/2$ in $\mathbb{H}^2\times\R$, and the periodicity of these surfaces is with respect to a discrete group of translations along a geodesic that is neither vertical nor horizontal in the Riemannian product $\mathbb{H}^2\times\R$. Via the Daniel correspondence we are able to reduce this existence problem to a uniqueness problem in the Berger spheres: if a pair of linked horizontal geodesics bounds exactly two embedded minimal annuli (for a fixed orientation of the boundary curves) then tilted unduloids in $\mathbb{H}^2\times\R$ exist.
\end{abstract}
\maketitle
\section*{Introduction}
Constant mean curvature surfaces (for short \mcH-surfaces) have a rich history. Various existence and uniqueness results for \mcH-surfaces with certain topologies or under assumptions like compactness, embeddedness or properness are known in Euclidean space $\R^3$. For example, for each $H>0$, Delaunay constructed embedded surfaces of revolution with constant mean curvature $H$. Fixing $H>0$ and up to motion of Euclidean three-space, these so-called \emph{unduloids} form a one-parameter family interpolating between a degenerate chain of spheres and a cylinder. Korevaar, Kusner and Solomon proved them to be unique among properly embedded \mcH-annuli \cite{KKS}.

In the present paper we study \mcH-annuli in the Riemannian product $\mathbb{H}^2\times\R$. The space $\mathbb{H}^2\times\R$ is not isotropic and so for the suitably defined axis of our annuli we can distinguish between vertical, horizontal and tilted geodesics: Tilted geodesics admit no isometric rotations fixing them, while horizontal geodesics admit a half-turn rotation, and vertical geodesics have arbitrary rotations. 
Nevertheless, we have a family of isometric translations fixing each of these geodesics. Note that if $h\in\R$ and $\Gamma$ is a geodesic in $\mathbb{H}^2$, then reflections in horizontal planes $\mathbb{H}^2\times\{h\}$ and vertical planes $\Gamma\times\R$ are also isometries of $\mathbb{H}^2\times\R$.

Let us recall the three known examples of properly (Alexandrov) embedded \mcH-annuli with $H>1/2$ in $\mathbb{H}^2\times\R$. \emph{Vertical unduloids} have been constructed by Hsiang and Hsiang in \cite{Hsiang} as \mcH-surfaces invariant under rotations about a vertical geodesic; \emph{horizontal unduloids}, that is, surfaces invariant under half-turn rotations about a horizontal geodesic, were constructed by Manzano and Torralbo in \cite{ManzanoTorralbo} via a conjugate Plateau construction. Each of these examples is periodic along a vertical respectively horizontal geodesic, and is (up to an isometry) part of a one-parameter family degenerating to a cylinder and to a chain of spheres---as is the case for unduloids in $\R^3$. Finally, \emph{tilted cylinders}, invariant under translations along a tilted geodesic have been constructed by Onnis in \cite{Onnis} (see also the author's Ph.D. thesis \cite{VrzinaPhD}).

Open is the existence of properly embedded \mcH-annuli which are periodic but not translation invariant with respect to a tilted geodesic. 

\begin{Conjnn}Let $H>\frac{1}{2}$ and let $\gamma$ be a geodesic in $\mathbb{H}^2\times\R$ with slope $\alpha\in[0,\pi/2]$ relative to a vertical geodesic. Then, up to an isometry of $\mathbb{H}^2\times\R$, there exists a one-parameter family $\left(\Sigma^{c,\alpha}\right)_{c\in(0,1]}$ of properly (Alexandrov) embedded \mcH-annuli which are singly periodic with respect to translations along $\gamma$; we refer to them as \emph{unduloids with axis $\gamma$}. If $\alpha\in(0,\pi/2)$ we call the surfaces \emph{tilted unduloids}.\index{tilted unduloid}\end{Conjnn}
The parameter $c$ corresponds to the \emph{neck size}\index{neck size} of $\Sigma^{c,\alpha}$, that is, $c$ parametrises the length of the shortest closed geodesic on $\Sigma^{c,\alpha}$. The limiting cases should be: $\Sigma^{0,\alpha}$, a degenerate chain of \mcH-spheres aligned along $\gamma$, and $\Sigma^{1,\alpha}$, a cylinder with tilted axis. See \autoref{tiltedcylinderunduloid} for a qualitative sketch of these surfaces. We remark that $H=1/2$ is the \emph{critical mean curvature} in $\mathbb{H}^2\times\R$, which means that compact immersed \mcH-surfaces only exist for $H>1/2$.



\begin{center}
	\begin{figure}
		\includegraphics[width=1\linewidth]{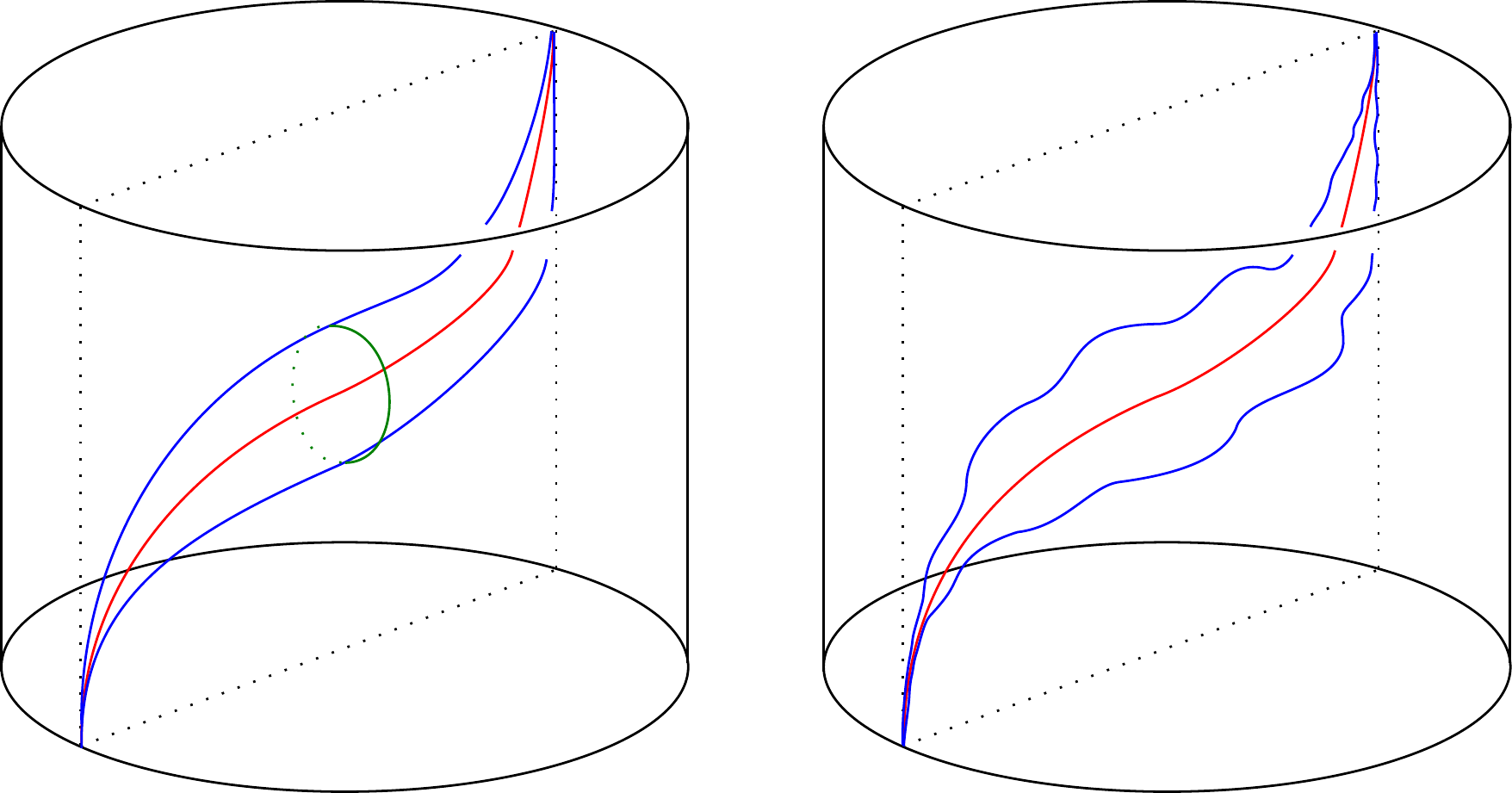}
		\caption{Left: A tilted cylinder in $\mathbb{H}^2\times\R$ is invariant under translations along the geodesic $\gamma$ shown in red: the simple closed curve shown in green generates the cylinder. The vertical plane indicated with dots contains $\gamma$ and intersects the cylinder in the geodesics shown in blue.\\Right: Conjectured intersection pattern of a tilted unduloid with a vertical plane. The blue curves are periodic.}
		\label{tiltedcylinderunduloid}
	\end{figure}
\end{center}

%
%


With the present paper we suggest to study the conjecture with the tool of the Daniel correspondence. We do not prove the conjecture, but we show two results which reduce the existence problem for tilted unduloids in $\mathbb{H}^2\times\R$ to a uniqueness problem for embedded minimal annuli bounded by linked horizontal geodesics in a Berger sphere.


To obtain these results we first study necessary conditions for the existence of unduloids in $\mathbb{H}^2\times\R$. Alexandrov reflection as in \cite{KKS} shows that an unduloid in $\mathbb{H}^2\times\R$ has a vertical mirror plane separating it into two halves. Each half is an \mcH-surface with two boundary curves contained in the plane, as illustrated in \autoref{tiltedcylinderunduloid}. A method for the construction of \mcH-surfaces with symmetries is the conjugate Plateau construction. It is based on the Daniel correspondence: One half of an unduloid in $\mathbb{H}^2\times\R$ is a disk corresponding to a minimal surface bounded by a pair of horizontal geodesics in a Berger sphere; these horizontal geodesic circles must be covered infinitely often. This follows from the Daniel correspondence, which we review in \autoref{section1} alongside our notation for the Berger spheres.

It is tempting to guess that the conjectured two-parameter family of unduloids in $\mathbb{H}^2\times\R$ corresponds to a two-parameter family of horizontal geodesics as boundary of a uniquely determined minimal surface in a Berger sphere. However, the problem turns out to be more involved as the boundary circles bound several minimal surfaces. In general, for a conjugate Plateau construction, we need to answer the following questions: 1. What are the correct boundary contours? 2. What kind of minimal surface bounded by these contours are we looking for specifically? The answers are given in \autoref{section2} as the first result of this paper. 

In order to answer these questions we revisit the known examples of unduloids in $\mathbb{H}^2\times\R$ whose axes are vertical or horizontal. In case of one half of a horizontal unduloid the corresponding horizontal geodesics in the Berger spheres are linked and must have---what we call---linearly dependent horizontal fields. This answers the first question:
\begin{enumerate}
	\item[] The pair of horizontal geodesics in the Berger sphere are linked and must have linearly independent horizontal fields.
\end{enumerate}
These contours are unrelated to the angle $\alpha$ of the axis. Indeed, any pair of linked horizontal geodesics in the Berger spheres bounds an embedded minimal annulus, namely a \emph{spherical helicoid}, also introduced in \autoref{section1}. The corresponding surface in $\mathbb{H}^2\times\R$ is a vertical unduloid. Therefore, giving the answer to problem 2, we can state:

\begin{enumerate}
	\item[{}] The desired surface is an embedded minimal annulus bounded by these great circles different from the spherical helicoid (which corresponds to vertical unduloids).
\end{enumerate}
How can multiple solutions be obtained? One option is to construct them explicitly, which is possible for linearly dependent horizontal fields. In \autoref{examplehorizontal} we prove an explicit multiple solution theorem:  there is a one-parameter family of linked horizontal geodesics with linearly dependent horizontal fields that bounds two minimal annuli, one corresponding to one half of a horizontal unduloid and the other one to a piece of a vertical unduloid. These minimal annuli are in fact embedded. The case of linearly independent horizontal fields does not seem to admit such an explicit construction. That is why we discuss a minimax principle due to Min Ji which yields that any pair of linked horizontal geodesics in a Berger sphere bounds at least two minimal annuli, a ``min`` and a ``minimax``. This minimax principle does not give information on the embeddedness of the solutions.

The embedded minimal annulus we are interested in is not explicit and thus we cannot control the geometry of the conjugate \mcH-surface in $\mathbb{H}^2\times\R$. We overcome this difficulty by employing a continuity method in \autoref{section3}. This is done in two steps. We first formulate three hypotheses (H1) to (H3) for an embedded minimal annulus bounded by linked horizontal geodesics. In \autoref{tiltedsolution} we show that these hypotheses are sufficient to yield a tilted unduloid in $\mathbb{H}^2\times\R$. The proof of this lemma only assumes basic knowledge about the Daniel correspondence and is independent from most of the details of \autoref{section2}. Then we discuss how these hypotheses can be satisfied, based on the contours determined in \autoref{section21}. The sister surface of one half of a horizontal unduloid (or of a horizontal cylinder) naturally satisfies (H1) to (H3). The hypotheses are preserved under slight continuous deformations of this sister surface. If we assume that every pair of linked horizontal geodesics bounds exactly two embedded minimal annuli these continuous deformations exist, and in \autoref{perturbationresult} we prove the existence of tilted unduloids as a perturbation of horizontal unduloids.



We note that uniqueness is an open problem: Meeks conjectured that properly (Alexandrov) embedded \mcH-annuli with $H>1/2$ are unduloids in $\mathbb{H}^2\times\R$; see \cite[Conjecture 4.22 (2)]{MP}. This would be the natural generalisation from \cite{KKS} to $\mathbb{H}^2\times\R$. In fact, this initiated the study of the existence problem in this paper.

\subsubsection*{Acknowledgement}This work extends upon a part of the author's Ph.D. thesis at TU Darmstadt. The author would like to thank his advisor Karsten Gro\ss{}e-Brauckmann for guidance and suggestions throughout the preparation of this paper, Rob Kusner for helpful discussions, and a referee for useful comments that helped clarifying important aspects.

\section{Preliminaries on Berger spheres and singly periodic surfaces}\label{section1}
In order to discuss the Daniel correspondence which relates surfaces with constant mean curvature $H>1/2$ in $\mathbb{H}^2\times\R$ to minimal surfaces in the Berger spheres, let us first review the Berger spheres. Moreover we will discuss periodicity and study symmetries of singly periodic \mcH-annuli.
\subsection{Geometry of the Berger spheres}\label{chapter31}
In the following let $\kappa$ and $\tau$ be real numbers with $\kappa>0$ and $\tau\neq0$. In this case the simply connected homogeneous $3$-manifold $E(\kappa,\tau)$ is compact. We discuss a model and geometric properties.

\subsubsection{Model of Berger spheres and isometries}Consider the standard three-sphere \[\mathbb{S}^3:=\{p\in\R^4\colon\lvert p\rvert_{\R^4}^2=1\}\] and let $V$ be the matrix
\[
V=\begin{pmatrix}0&-1&0&0\\
1&0&0&0\\
0&0&0&1\\
0&0&-1&0                                                                                                                                                      \end{pmatrix},
\]
which induces the vector field
\[
V(a,b,c,d)=(-b,a,d,-c)
\]
on $\mathbb{S}^3$. We endow $\mathbb{S}^3$ with the metric
\begin{equation}g_{\kappa,\tau}(X,Y):=\frac{4}{\kappa}\left[\langle X,Y\rangle_{\R^4}+\bigg(\frac{4\tau ^2}{\kappa}-1\bigg)\langle X,V\rangle_{\R^4}\langle Y,V\rangle_{\R^4}\right]\label{eq:bergermetric}.\end{equation}\index{Berger metric}
A \emph{Berger sphere}\index{Berger sphere} is the Riemannian space $\mathbb{S}^3(\kappa,\tau):=(\mathbb{S}^3,g_{\kappa,\tau})$. Up to scaling of the metric, the Berger spheres form a one-parameter family of spaces: If we set $\eta:=\frac{4\tau^2}{\kappa}$ and scale $g_{\kappa,\tau}$ by $\frac{\kappa}{4}$ the metric depends only on $\eta\in(0,\infty)$.



The isometries of $\mathbb{S}^3(\kappa,\tau)$ can be described as follows, see \cite[Section 2]{Torralbo2}:
\begin{equation}
\operatorname{Iso}\left(\mathbb{S}^3(\kappa,\tau)\right)=\begin{cases}
\mathsf{O}(4)&\mbox{if }\kappa=4\tau^2,\\
\{A\in\mathsf{O}(4)\colon AV=\pm VA\}&\mbox{if }\kappa\neq4\tau^2.
\end{cases}\label{isometry}
\end{equation}

\subsubsection{Structure of a metric Lie group}The quaternions\index{quaternions} is the skew-field $\R^4$ considered with basis
$\bone$, $\ii$, $\jj$, $\kk$ and bilinear product satisfying the relations 
\[
\bone\ii=\ii,\quad \ii^2=\jj^2=\kk^2=-\bone,\quad \ii\jj\kk=-\bone.
\]
It is well-known that this bilinear product induces a group structure on $\mathbb{S}^3$. For $p\in\mathbb{S}^3$ we consider the map
\[
\mathcal{L}_p\colon\mathbb{S}^3\to\mathbb{S}^3,\qquad 
\mathcal{L}_p(q):=pq.
\]
One can check that $\mathcal{L}_p$ is in $\mathsf{O}(4)$ and satisfies $\mathcal{L}_pV=V\mathcal{L}_p$, that is,  in view of \eqref{isometry} the map $\Lcal_p$ is an isometry of $\mathbb{S}^3(\kappa,\tau)$ and thus a Berger sphere is a metric Lie group.


\subsubsection{Identification with other models}The quaternions have a representation as complex $2\times2$ matrices since
\[
\R^4\to \mathsf{M}_2(\C),\qquad a\bone+b\ii+c\jj+d\kk\mapsto\begin{pmatrix}
a+bi&c+di\\
-c+di&a-bi
\end{pmatrix}
\]
is an injective ring homomorphism. The image of $\mathbb{S}^3\subset\R^4$ is the group $\mathsf{SU}(2)$. Identifying $\mathsf{SU}(2)$ with $\mathbb{S}^3=\{(z,w)\in\C^2\colon\lvert z\rvert^2+\lvert w\rvert^2=1\}$ the group structure is as follows:
\[
(z_1,w_1)(z_2,w_2)=(z_1z_2-w_1\widebar{w_2},z_1w_2+w_1\widebar{z_2})\quad\text{ and }\quad (z,w)^{-1}=(\widebar{z},-w).
\]
The neutral element is $(1,0)$ and left translations have the following representation:
\[
\Lcal_{(z_1,w_1)}\colon\mathbb{S}^3(\kappa,\tau)\to\mathbb{S}^3(\kappa,\tau),\qquad\Lcal_{(z_1,w_1)}(z_2,w_2):=(z_1,w_1)(z_2,w_2).
\]

\subsubsection{Orthonormal frame and Hopf fibration}At the identity $\bone\in\mathbb{S}^3$ the vectors $\jj$, $\kk$ and $\ii$ are orthogonal. Since left translations are isometries we obtain a global orthonormal frame by setting
\begin{align}E_1(p):=\frac{\sqrt{\kappa}}{2}\Lcal_{p}(\jj)&=\frac{\sqrt{\kappa}}{2}(-c,-d,a,b)=\frac{\sqrt{\kappa}}{2}(-w,z),\nonumber\\
E_2(p):=\frac{\sqrt{\kappa}}{2}\Lcal_{p}(\kk)&=\frac{\sqrt{\kappa}}{2}(-d,c,-b,a)=\frac{\sqrt{\kappa}}{2}(iw,iz),\label{BergerONBleft}\\ 
\xi(p):=\frac{\kappa}{4\tau}\Lcal_{p}(\ii)&=\frac{\kappa}{4\tau} (-b,a,d-c)=\frac{\kappa}{4\tau}(iz,-iw),\nonumber\end{align}\index{orthonormal frame in $\mathbb{S}^3(\kappa,\tau)$}
where we identify $p=a\bone+b\ii +c\jj +d\kk \in\R^4$ with $p=(a+ib,c+id)=(z,w)\in\mathbb{C}^2$.


Using this orthonormal frame and the definition of the Riemannian metric on $\mathbb{S}^3(\kappa,\tau)$ we get the following expression of the Levi-Civita connection:
\[\begin{array}{lll}\nabla_{E_1}E_1=0,&\nabla_{E_1}E_2=\tau \xi ,&\nabla_{E_1}\xi=-\tau E_2,\\
\nabla_{E_2}E_1=-\tau \xi,& \nabla_{E_2}E_2=0,&\nabla_{E_2}\xi=\tau E_1,\\
\nabla_{\xi}E_1=\left(\frac{\kappa}{2\tau}-\tau\right) E_2,&\nabla_{\xi}E_2=-\left(\frac{\kappa}{2\tau}-\tau\right)E_1,& \nabla_{\xi}\xi=0.
\end{array} 
\]
We introduce the Hopf fibration:
\begin{Prop}\label{hopffibration}The \emph{Hopf fibration}\index{Hopf fibration}
	\begin{equation}\Pi\colon \mathbb{S}^3(\kappa,\tau)\to\mathbb{S}^2(\kappa),\qquad \Pi(z,w):=\frac{1}{\sqrt{\kappa}}\left(-2izw,\lvert z\rvert^2-\lvert w\rvert^2\right),\label{hopffibrationleft}\end{equation}
	where $\mathbb{S}^2(\kappa)$ denotes the two-sphere of radius $\frac{1}{\sqrt{\kappa}}$, is a Riemannian submersion whose fibres are geodesics. The vertical unit Killing vector field is given by $\xi$. The horizontal space of this submersion is spanned by $E_1$ and $E_2$. Moreover, the base space $\mathbb{S}^2(\kappa)$ has constant sectional curvature $\kappa$ and the bundle curvature of the fibration is $\tau$.
\end{Prop}
A proof of this proposition is an elementary computation in terms of quaternions.
\begin{Def}We call $\xi$ a \emph{Hopf field}\index{Hopf field}. For $a,b\in\R$ with $a^2+b^2=1$ we call the linear combination $F=aE_1+bE_2$ a \emph{horizontal field}\index{horizontal field}. An integral curve of $\xi$ or $F$ is called a \emph{Hopf circle}\index{Hopf circle} or a \emph{(horizontal) $F$-circle}\index{horizontal circle}, respectively.\end{Def}
\subsubsection{Vertical and horizontal geodesics}\label{subgeodesics}Let $p=(z,w)$ be in $\mathbb{S}^3(\kappa,\tau)$. Then
\begin{equation}v(s):=p\cos\left(\frac{\kappa}{4\tau}s\right)+ \frac{4\tau}{\kappa}\xi(p)\sin\left(\frac{\kappa}{4\tau}s\right)=\begin{pmatrix}\exp\left(i\frac{\kappa}{4\tau}s\right)z\\
\exp\left(-i\frac{\kappa}{4\tau}s\right)w                                                                                                  \end{pmatrix}
\label{vgeodesic}\end{equation}\index{vertical geodesic}\index{Hopf circle}
parametrises a vertical unit-speed geodesic through $p$. Since $\nabla_{\xi}\xi=0$, this claim about $v$ follows from $\xi(v(s))=v^\prime(s)$, that is, $v$ is an integral curve of $\xi$. It is a fibre since it projects to the point $\Pi(v(s))=\frac{1}{\sqrt{\kappa}}(-2izw,\lvert z\rvert^2-\lvert w\rvert^2)$.

Given the horizontal field $F=F_\varphi=\cos(\varphi)E_1+\sin(\varphi)E_2$, the curve \begin{align}h(t)&:=p\cos\left(\frac{\sqrt{\kappa}}{2}t\right)+ \frac{2}{\sqrt{\kappa}}F_\varphi(p)\sin\left( \frac{\sqrt{\kappa}}{2}t\right)\nonumber\\
&=\begin{pmatrix}\cos\left(\frac{\sqrt{\kappa}}{2}t\right)z-\sin\left(\frac{\sqrt{\kappa}}{2}t\right)\exp(-i\varphi)w\\
\cos\left(\frac{\sqrt{\kappa}}{2}t\right)w+\sin\left(\frac{\sqrt{\kappa}}{2}t\right)\exp(i\varphi)z                                      \end{pmatrix}
\label{hgeodesic}\end{align}
parametrises a horizontal unit-speed geodesic through $p$ with tangent vector $F_\varphi$. This is a consequence of $\nabla_{F_\varphi}F_\varphi=0$ and $F_\varphi(h(t))=h^\prime(t)$.\index{horizontal geodesic}

We read off the lengths of the Hopf circle\index{Hopf circle} $v$ and of the horizontal $F$-circle\index{horizontal circle} $h$.
\begin{Obs}Vertical and horizontal geodesics have the following respective lengths:
	\begin{equation}
	\operatorname{length}(v)=\frac{8\tau\pi}{\kappa}\quad\mbox{ and }\quad\operatorname{length}(h)=\frac{4\pi}{\sqrt{\kappa}}.\label{lengthgeodesic}
	\end{equation}\end{Obs}

\subsubsection{An important minimal surface: the spherical helicoid}Let $h_1$ and $h_2$ be horizontal geodesics in $\mathbb{S}^3(\kappa,\tau)$. It is natural to ask the following question: Is there an (embedded) minimal surface bounded by $h_1$ and $h_2$? 

The next proposition establishes an explicit solution as long as $h_1$ and $h_2$ are either identical or linked. If they are linked then they bound a \emph{spherical helicoid}\index{spherical helicoid} with a vertical axis $v$ joining $h_1$ and $h_2$; the rulings are horizontal geodesics which rotate (with constant angular speed) about the axis. Letting the pitch of such a helicoid go to $0$ we arrive at the case that $h_1$ and $h_2$ are identical. Then $h_1=h_2$ bounds a so-called \emph{horizontal umbrella}\index{horizontal umbrella}.

For the upcoming sections it is useful to describe these helicoids explicitly:
\begin{Prop}\label{bergerhelicoid}Let $h_1$ and $h_2$ be identical or linked horizontal geodesics in $\mathbb{S}^3(\kappa,\tau)$. 
\begin{enumerate}[(i)]
 \item There are $\ell\in(0,8\tau\pi/\kappa)$ and $\varphi\in[0,2\pi)$ such that the parametrisation
	\begin{equation}
	f\colon\R\times[0,1]\to\mathbb{S}^3(\kappa,\tau),\qquad f(x,y):=\begin{pmatrix}\cos\left(\frac{\sqrt{\kappa}}{2}x\right)\exp\left(\pm i \frac{\kappa}{4\tau}\ell y\right)\\
	\sin\left(\frac{\sqrt{\kappa}}{2}x\right)\exp\Big(i\left(\varphi \pm\frac{\kappa}{4\tau}\ell\right)y\Big)\end{pmatrix}\mbox{,}\label{helicoidpar}
	\end{equation}
defines, up to an ambient isometry, an immersed minimal annulus bounded by the linked horizontal geodesics $h_1=f(\cdot,0)$ and $h_2=f(\cdot,1)$. We call $\ell$ the \emph{pitch} and $\varphi\in[0,2\pi]$ the \emph{angle} of the \emph{spherical helicoid}. For $\ell=0$ we choose $\varphi=\pi$ and obtain an immersed \emph{horizontal umbrella} on $(0,\pi/\sqrt{\kappa})\times[0,1]$.

\item The parametrisation \eqref{helicoidpar} defines an embedding on $[0,2\pi)\times[0,1]$ respectively on $[0,\pi/\sqrt{k})\times[0,1]$ if and only if $\ell\in(0,4\tau\pi/\kappa]$ respectively $\ell=0$. 

\item For any choice of $h_1$ and $h_2$ there exists $\ell$ as in (ii), and for prescribed orientations of $h_1$ and $h_2$ the embedded spherical helicoid bounded by $h_1$ and $h_2$ is unique. Changing the orientation of either $h_1$ or $h_2$ gives another embedded minimal annulus.
\end{enumerate}
\end{Prop}
\begin{proof}[Sketch of proof](i): The curves $\Pi\circ h_1$ and $\Pi\circ h_2$ are geodesic in $\mathbb{S}^2(\kappa)$ and thus they intersect in at least two points. Therefore $h_1$ and $h_2$ are joined by a segment of a vertical geodesic with length $\ell\in[0,8\tau\pi/\kappa)$. After a left translation we may assume that $h_1$ is a horizontal geodesic through $p=(1,0)$ and that the vertical segment $v$ emanates from $h_1(0)=p=(1,0)$ as well. After rotation about the vertical geodesic $v$, which are isometries in $E(\kappa,\tau)$-spaces, we can assume $h_1$ has $E_1$ as horizontal field. Thus $h_2$ is a horizontal geodesic $h_2$ through $v(\pm \ell)$ with a horizontal field $F_\varphi=\cos(\varphi)E_1+\sin(\varphi)E_2$ for some $\varphi\in[0,2\pi)$. From these information one can compute $h_1$, $v$ and $h_2$ explicitly by using \eqref{hgeodesic} and \eqref{vgeodesic}. Finally one uses \eqref{isometry} to check that, for $y\in\R$, the mapping
 \[
  \Psi_y\colon\mathbb{S}^3(\kappa,\tau)\to\mathbb{S}^3(\kappa,\tau)\mbox{,}\qquad\Psi_y(z,w):=\begin{pmatrix}z\exp\left(\pm i \frac{\kappa}{4\tau}\ell y\right)\\
 w\exp\Big(i\left(\varphi \pm\frac{\kappa}{4\tau}\ell\right)y\Big)\end{pmatrix}
 \]
is an isometry and satisfies $f(x,y)=\Psi_y(h_1(x))$. By computation one verifies that $f$ is an immersion on the domains stated in (i).

(ii): We have $f(x,y)=\Psi_y(h_1(x))$ and $\Psi_y(h_1(0))=f(0,y)=:v(y)$ by construction of $f$ in (i). Assume $\ell\geq 4\tau\pi/\kappa$. Then there is $y_0\in[0,1]$ such that $v\vert_{[0,y_0]}$ has length $4\tau\pi/\kappa$. The length of a vertical geodesic is $8\tau\pi/\kappa$, so that starting from $h_1(0)$ the vertical geodesic $v$ meets $h_1$ again at $v(y_0)=h_1(2\pi/\sqrt{\kappa})$. Thus $f$ cannot be embedded for $\ell\geq 4\tau\pi/\kappa$.

For embeddedness of $f$ on $[0,4\pi/\sqrt{\kappa})\times[0,1]$ in the case $\ell\in(0,4\tau\pi/\sqrt{\kappa})$ we note $F_y=\frac{\partial f}{\partial x}(x,y)=\cos(\varphi y)E_1+\sin(\varphi y)E_2$ for $y\in[0,1]$. This shows that $f$ is foliated by horizontal geodesics with different horizontal fields. Thus $f$ cannot have self intersections. For $\ell=0$ one argues similarly.

(iii): If we have $\ell\in[4\tau\pi/\kappa,8\tau\pi/\kappa)$ we replace it by $\tilde{\ell}:=8\tau\pi/\kappa-\ell\in[0,4\tau\pi/\kappa)$ and change the orientation of $v$. This determines $\ell$ uniquely. 

Finally we consider the claim about the change of orientation for $h_1$ and $h_2$. The surface $(x,y)\mapsto f(-x,y)$ is bounded by $x\mapsto h_1(-x)$ and $x\mapsto h_2(-x)$. Therefore it is sufficient to only change the orientation of $h_2$. This is done by replacing $\varphi$ with $\varphi\pm \pi\in[0,2\pi)$. 
\end{proof}

Let us discuss uniqueness of the Plateau problem for $h_1$ and $h_2$ as well as further non-embedded solutions:
\begin{Rem}\label{immersedunique}
	We have already pointed out that joining $h_1$ and $h_2$ with a vertical axis of length $\ell\geq4\tau\pi/\kappa$ yields a non-embedded minimal annulus. In fact, instead of joining $h_1$ and $h_2$ with the vertical axis of length $\ell\in[0,4\tau\pi/\kappa)$, we can traverse the great circle containing this axis several times and then join $h_1$ and $h_2$ with the axis of length $\ell$. Thus each $\ell_n=\ell+n8\tau\pi/\kappa$ with $n\in\Z\setminus\{0\}$ and $\ell\in[0,4\tau\pi/\kappa)$ defines an immersed minimal annulus bounded by the same horizontal geodesics $h_1$ and $h_2$. Therefore we do not have uniqueness or finite number of solutions among immersed minimal annuli bounded by $h_1$ and $h_2$. These solutions do not occur as solutions of the Plateau problem for $h_1$ and $h_2$ since they are not area-minimising.	
\end{Rem} 

\subsection{Daniel and Lawson correspondence}\label{chapter32}We introduce the \emph{Daniel correspondence} by Daniel from \cite{Daniel} and its properties \`{a} la Manzano and Torralbo in \cite{ManzanoTorralbo} and \cite{Torralbo2}. We refer to \cite[Section 2 and Section 3]{KarstenProceeding} for the \emph{Lawson correspondence}. Most results we state in this section are quotations from these papers.

Fist we introduce some notation:
\begin{Def}Let $\Sigma$ be an oriented surface immersed into some Riemannian manifold $E$. Then the Riemannian metric $\langle\cdot,\cdot\rangle$ of $E$ induces a rotation of angle $\frac{\pi}{2}$ on the tangent bundle to $\Sigma$. We denote this rotation by $J$. The Levi-Civita connection $\nabla$ of $E$ defines the \emph{shape operator}\index{shape operator} $S$ by $SX:=-\nabla_X N$ where $X$ is tangent to $\Sigma$ and $N$ is a unit normal vector field on $\Sigma$.\end{Def}

The $E(\kappa,\tau)$-spaces are simply connected homogeneous three-manifolds diffeomorphic to $\R^3$, $\mathbb{S}^3$ or $\mathbb{S}^2\times\R$ and arise as Riemannian fibrations $E\to M$ with geodesic fibres, where $M=M(\kappa)$ has curvature $\kappa\in\R$ and the bundle curvature is $\tau\in\R$. We denote the vertical field tangent to the fibres by $\xi$. Constructing \mcH-surfaces with $H\neq0$ and boundary in vertical or horizontal planes of a product manifold $E(\kappa,0)=M(\kappa)\times\R$ is difficult, because it is a free boundary problem. The Daniel correspondence\index{Daniel correspondence} reduces this free boundary problem to a fixed boundary problem, but one has to deal with more complicated ambient spaces:

\begin{Prop}[special case of {{\cite[Theorem 5.2]{Daniel}}}]\label{danielcorr}Let $\kappa$ and $H$ be in $\R$. For each simply connected
	\[\text{minimal surface }\Sigma\subset E(\kappa+4H^2,H)\text{ with shape operator }S\]
	there exists an isometric
	\begin{equation}
	\text{\mc} H\text{-surface }\widetilde{\Sigma}\subset E(\kappa,0)=M(\kappa)\times\R \text{ with shape operator }\widetilde{S}=JS+H\id
	\label{eq:daniel}
	\end{equation}
	and vice versa.
\end{Prop}
We will refer to the surfaces $\Sigma$ and $\widetilde{\Sigma}$ as \emph{conjugate sister surfaces}\index{conjugate sister surface}, or simply as \emph{sisters}\index{sister surface}.

Let us take a look at some special choices for $\kappa$ and $H$:

\begin{Ex}\label{exdaniel} \begin{enumerate}[(a)]
		\item For $\kappa=0$ and $H=1$ we obtain Lawson's correspondence of minimal surfaces in $\mathbb{S}^3$ and \mcone-surfaces in $\R^3$.
		
		\item \mcH-surfaces $\widetilde{\Sigma}$ in the space $E(-1,0)=\H^2\times\R$ with $H>\frac{1}{2}$ have minimal sisters $\Sigma$ in Berger spheres $E(4H^2-1,H)$ since the base curvature $4H^2-1$ is positive.
	\end{enumerate}
\end{Ex}
Daniel's correspondence has the following \emph{first order description}\index{first order description}:
\begin{Prop}[{{\cite[Theorem 5.2]{Daniel}}}]\label{firstorderprop}Suppose an \mcH-immersion $\tilde{f}\colon\Omega\to E(\kappa,0)$ and minimal immersion ${f\colon\Omega\to {E(4H^2+\kappa,H)}}$ parametrise sister surfaces with unit normal fields $\widetilde{N}$ and $N$, respectively. Let $\widetilde{Z}$ and $Z$ be the corresponding restrictions of the vertical vector fields $\widetilde{\xi}$ and $\xi$ along $\tilde{f}$ and $f$, respectively. Denote the projections to the respective tangents spaces of these vector fields by $\widetilde{T}$ and $T$. Then
	\begin{equation}
	\big\langle \widetilde{Z},\widetilde{N}\big\rangle=\langle Z,N\rangle\quad\mbox{ and }\quad df^{-1}(T)=Jd\tilde{f}^{-1}\big(\widetilde{T}\big).\label{firstorderekt}
	\end{equation}
\end{Prop}
In an $E(\kappa,\tau)$-space, a rotation of angle $\pi$ about a horizontal or vertical geodesic is an isometry. In the product spaces $E(\kappa,0)=M(\kappa)\times\R$, reflections in vertical and horizontal planes are isometries. We refer to these planes as \emph{mirror planes}\index{mirror plane}. A curve $\tilde{c}$ on a surface $\tilde{\Sigma}$ in $E(\kappa,0)$ is called \emph{mirror curve}\index{mirror curve} if it is contained in a mirror plane and its conormal $\tilde{\eta}$ is perpendicular to the mirror plane.

The correspondence relates mirror curves as follows:

\begin{Prop}[{{\cite[Proposition 3]{Torralbo2}}}]\label{symmetryhom}Let $\widetilde{\Sigma}$ be an \mcH-surface in $M(\kappa)\times\R$ with sister minimal surface $\Sigma$ in $E(4H^2+\kappa,H)$.
	\begin{enumerate}[(i)]
		\item A curve $\tilde{c}$ on $\widetilde{\Sigma}\subset M(\kappa)\times\R$ is a vertical mirror curve if and only if its sister curve $c$ on $\Sigma\subset E(\kappa+4H^2,H)$ is a horizontal (ambient) geodesic.\index{horizontal geodesic}
		
		\item Similarly, $\tilde{c}$ is contained in a horizontal mirror curve if and only if $c$ on the minimal sister $\Sigma\subset E(4 H^2+\kappa,H)$ is a vertical geodesic (contained in a fibre).\index{vertical geodesic}
	\end{enumerate}\end{Prop}
	
	Another issue is the smooth extension of surfaces bounded by geodesics. If a minimal surface $\Sigma$ in some $E(\kappa,\tau)$-space has a vertical or horizontal geodesic $c$ contained in $\partial\Sigma$ then it is possible to extend $\Sigma$ by geodesic reflection $\rho$ around $c$. This is better known as \emph{Schwarz reflection}\index{Schwarz reflection} and the extension is smooth; for details we refer to \cite[Section 2.2]{ManzanoTorralbo}. Let $f\colon\Omega\to E(4H^2+\kappa,H)$ parametrise $\Sigma\cup\rho(\Sigma)$. Then $c$ is contained in $f(\Omega)$, and the Daniel correspondence also relates such extensions:
	
	\begin{Prop}[{{\cite[Lemma 2]{ManzanoTorralbo}}}]\label{extensionhom}Let an \mcH-immersion $\tilde{f}\colon\Omega\to E(\kappa,0)$ and a minimal immersion ${f\colon\Omega\to {E(4H^2+\kappa,H)}}$ parametrise sister surfaces. If $f(\Omega)$ is invariant by a $\left\{\begin{smallmatrix}\mbox{horizontal}\\\mbox{vertical}\end{smallmatrix}\right\}$ geodesic reflection, then $\tilde{f}(\Omega)$ is invariant by a reflection in a $\left\{\begin{smallmatrix}\mbox{vertical}\\\mbox{horizontal}\end{smallmatrix}\right\}$ plane. 
	\end{Prop}

	%
	%
	In the special case of the three-sphere, i.e. $\kappa=4$ and $\tau=1$, we have more detailed information on the sister curves available:
	\begin{Rem}As highlighted earlier, the case $\kappa=0$ and $H=1$ is the Lawson correspondence of \mcone-surfaces in $\R^3$ and minimal surfaces in $\mathbb{S}^3$. Here, mirror curves in a vertical mirror plane of $\R^3$ correspond to horizontal geodesics with the same horizontal field $F$; see for example \cite[Corollary 3.1]{KarstenProceeding}. Comparing this fact with \autoref{symmetryhom}, a natural question arises: Let $\tilde{c}$ be a vertical mirror curve of an \mcH-surface $\widetilde{\Sigma}\subset \mathbb{H}^2\times\R$ with $H>\frac{1}{2}$. Is the sister curve $c$ on $\Sigma\subset\mathbb{S}^3(4H^2-1,H)$ the integral curve of a fixed horizontal field $F$?
		
		In \autoref{vundushape} (iii) we show that the two vertical mirror curves of one half of a vertical unduloid in $\mathbb{H}^2\times\R$ correspond to a $F_1$-circle and $F_2$-circle with $F_1\neq \pm F_2$. In this sense \cite[Corollary 3.1]{KarstenProceeding} does not hold for the Daniel correspondence. The reason is that the Lawson correspondence of \mcone-surfaces in $\R^3$ and minimal surfaces in $\mathbb{S}^3$ admits a first order description for any Killing field, not only for the vertical Killing field as it is the case in the Daniel correspondence.\end{Rem} 
	
	\subsection{Singly periodic surfaces in \texorpdfstring{$\mathbb{H}^2\times\R$}{H2xR}: Basic definitions and properties}\label{chapter33}
	
	We first need to define certain translations in $\mathbb{H}^2\times\R$:
	\begin{Def}[Translation induced by geodesic]
		Let $\tilde{\gamma}$ be a geodesic in $\mathbb{H}^2\times\R$ with slope $\alpha$, that is, $\cos\alpha\equiv\langle \tilde{\xi}\circ\tilde{\gamma},\tilde{\gamma}^\prime\rangle$. The projection $\Pi\circ\tilde{\gamma}$, assumed to be parametrised by arc-length, is a geodesic in $\mathbb{H}^2$ and induces a one-parameter family of hyperbolic translations $\psi_s\colon \mathbb{H}^2\to \mathbb{H}^2$ fixing $\Pi\circ\tilde{\gamma}$ as a set. Then \[\Phi_s\colon\mathbb{H}^2\times\R\to\mathbb{H}^2\times\R,\quad\Phi_s(p,h):=(\psi_{s\sin(\alpha)}(p),s\cos(\alpha)+h)\]
		is an isometry such that $\Phi_s(\tilde{\gamma}(0))=\tilde{\gamma}(s)$ for all $s\in\R$. We refer to $(\Phi_s)_{s\in\R}$ as \emph{translation along $\tilde{\gamma}$}.
	\end{Def}
	In case of $\mathbb{H}^2\times\R$ different geodesics $\tilde{\gamma}$ can induce the same one-parameter family $(\Phi_s)_{s\in\R}$. For instance, translations along any vertical geodesic in $\mathbb{H}^2\times\R$.
	
	\begin{Def}[Singly periodic]A surface $\widetilde{\Sigma}$ in $\mathbb{H}^2\times\R$ is called \emph{singly periodic}\index{singly periodic} if the following is satisfied: There is a one-parameter family of isometries $(\Phi_s)_{s\in\R}$, induced by translations along a geodesic $\tilde{\gamma}$, and a real number $T>0$ such that $\widetilde{\Sigma}$ is invariant under the discrete group $\{\Phi_{nT}\colon n\in\Z\}$. We call any geodesic $\tilde{\gamma}$ inducing the translations $(\Phi_s)_{s\in\R}$ an \emph{axis}\index{axis of a singly periodic surface} of $\widetilde{\Sigma}$.
	\end{Def}
	
	Two geodesics $\tilde{\gamma}_1$ and $\tilde{\gamma}_2$ in $\R^3$ generate $(\Phi_s)_{s\in\R}$ if and only if $\tilde{\gamma}_1^\prime=\tilde{\gamma}_2^\prime$. The same question in $\mathbb{H}^2\times\R$ has a more interesting answer, which we highlight along another important property of singly periodic surfaces:
	\begin{Prop}\label{singleprop}
		Let $\widetilde{\Sigma}$ be a singly periodic surface invariant under $(\Phi_{nT})_{n\in\Z}$.
		\begin{enumerate}[(i)]
			\item Suppose $\tilde{\gamma}_1$ and $\tilde{\gamma}_2$ both generate $(\Phi_s)_{s\in\R}$ in $\mathbb{H}^2\times\R$. Then either both of them are vertical geodesics or they lie in the same vertical plane and differ by a vertical translation.
			\item Let $\widetilde{\Sigma}$ be a singly periodic immersed annulus in $\mathbb{H}^2\times\R$. Then there exists a compact annulus $\widetilde{\Sigma}_0\subset\widetilde{\Sigma}$ such that $\widetilde{\Sigma}=\bigcup_{n\in\Z}\Phi_{nT}(\widetilde{\Sigma}_0)$ (the union is not necessarily disjoint). In particular $\widetilde{\Sigma}$ is a proper immersion.
		\end{enumerate}
	\end{Prop}
	\begin{proof}For (i) we distinguish two cases for $(\Phi_s)_{s\in\R}$. If $\tilde{\gamma}_1$ is a vertical geodesic, then $\Phi_s(x,y,z)=(x,y,z+s)$ and  $\tilde{\gamma}_2(s)=\Phi_s(\tilde{\gamma}_2(0))$ is a vertical geodesic, too. If $\tilde{\gamma}_1$ is non-vertical then $\Phi_s$ is the composition of a hyperbolic translation along a horizontal geodesic and a vertical translation. A hyperbolic isometry in $\mathbb{H}^2$ fixes exactly one geodesic and thus $(\Phi_s)_{s\in\R}$ fixes exactly one vertical plane. The geodesic $\tilde{\gamma}_2$ induces $(\Phi_s)_{s\in\R}$ as well and so it must lie in the same vertical plane. The geodesics $\tilde{\gamma}_1$ and $\tilde{\gamma}_2$ are orbits of $(\Phi_s)_{s\in\R}$ and thus they have the same slope with respect to the vertical Killing field $\widetilde{\xi}$ of $\mathbb{H}^2\times\R$. Therefore they differ by a vertical translation.
		
		For (ii) let $\Omega:=\{(x,y)\in\R^2\colon 0<x^2+y^2<1\}$ and $\tilde{f}\colon\Omega\to E(\kappa,0)$ be an immersion such that $\widetilde{\Sigma}=\tilde{f}(\Omega)$. Let $\alpha\colon[0,2\pi)\mbox{, }\alpha(t):=\frac{1}{2}(\cos(t),\sin(t))$. Then $\tilde{f}\circ\alpha$ has compact image and for $m\in\N$ sufficiently large $\tilde{f}\circ\alpha$ and $\Phi_{mT}\circ\tilde{f}\circ\alpha$ are disjoint due to the compactness of $\tilde{f}\circ\alpha$. Thus we find a curve $\beta$ in $\Omega$ that is disjoint from $\alpha$ and satisfies $\Phi_{mT}\circ\tilde{f}\circ\alpha=\tilde{f}\circ\beta$. As illustrated in \autoref{periodicannulus}, $\alpha$ and $\beta$ bound a compact annulus $\Omega_0$ in $\Omega$. The continuous image $\tilde{\Sigma}_0:=\tilde{f}(\Omega_0)$ is a compact annulus with the desired properties.\end{proof}
	\begin{center}
		\begin{figure}[ht]
			\def\svgwidth{0.3\linewidth}
			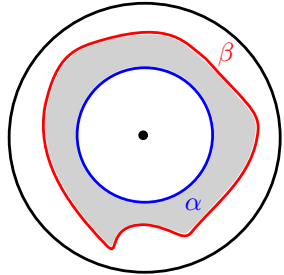
			\caption{Shaded in grey is the annulus $\Omega_0$ bounded by $\alpha$ and $\beta$}
			\label{periodicannulus}
		\end{figure}
	\end{center}
	
	We study symmetries of singly periodic (Alexandrov) embedded \mcH-annuli in the product $\mathbb{H}^2\times\R$. The main property is that such surfaces always have a vertical mirror plane:
	
	
	
	\begin{Prop}\label{singleclass}Let $\widetilde{\Sigma}$ be a singly periodic 
		(Alexandrov) embedded \mcH-annulus with axis $\tilde{\gamma}$ in $\H^2\times\R$. Then $\widetilde{\Sigma}$ has a vertical mirror plane $P$ such that $P$ separates $\widetilde{\Sigma}$ into simply connected \mcH-surfaces $\widetilde{\Sigma}_\pm$ with $\partial\widetilde{\Sigma}_\pm\subset P$ and $\widetilde{\Sigma}=\widetilde{\Sigma}_+\cup\widetilde{\Sigma}_-$.
		Moreover, if $\tilde{\gamma}$ is vertical then $\widetilde{\Sigma}$ is invariant under rotations about $\tilde{\gamma}$. If $\tilde{\gamma}$ is horizontal then the horizontal and the vertical plane containing $\tilde{\gamma}$ are mirror planes of $\widetilde{\Sigma}$.
	\end{Prop}
	
	\begin{proof}Here we only have vertical and horizontal planes at hand. For a vertical axis $\tilde{\gamma}$ the periodicity implies that $\widetilde{\Sigma}$ is contained in a vertical cylinder. Mazet has shown in \cite{Mazet} that such surfaces in $\mathbb{H}^2\times\R$ are rotationally invariant. 
		
		If $\tilde{\gamma}$ is non-vertical then its projection is a geodesic of $\mathbb{H}^2$. Let $\tilde{c}$ be a geodesic in $\mathbb{H}^2$ which intersects the projection of $\tilde{\gamma}$ orthogonally. The family of geodesics orthogonal to $\tilde{c}$ defines a family of vertical planes $(P_s)_{s\in\R}$ such that $\tilde{\gamma}$ is in $P_0$. Alexandrov's moving planes argument shows that $P_0$ is a mirror plane of $\widetilde{\Sigma}$. If $\tilde{\gamma}$ is horizontal then we can reason in the same way with respect to horizontal planes.\end{proof}

	Daniel's correspondence immediately implies the following:
	
	\begin{Cor}\label{singleh2r}Let $\widetilde{\Sigma}\subset\mathbb{H}^2\times\R$ be a singly periodic (Alexandrov) embedded \mcH-annulus with $H>\frac{1}{2}$. Then the minimal sister surface $\Sigma_\pm\subset\mathbb{S}^3(4H^2-1,H)$ of $\widetilde{\Sigma}_\pm$ is bounded by horizontal geodesics $h_1$ and $h_2$ (covered infinitely often).\end{Cor}

	
	
\section{Revisiting known examples}\label{section2}
Our goal is to construct \emph{tilted unduloids}\index{tilted unduloid} in $\mathbb{H}^2\times\R$, that is, singly periodic (Alexandrov) embedded \mcH-annuli with $H>\frac{1}{2}$ whose axis is neither vertical nor horizontal. To construct such surfaces we want to construct suitable minimal surfaces bounded by horizontal geodesics in $\mathbb{S}^3(4H^2-1,H)$, as indicated by \autoref{singleh2r}. We revisit the known examples in $\mathbb{H}^2\times\R$, vertical and horizontal unduloids, and it turns out it is a reasonable assumption to look for embedded minimal \emph{annuli} bounded by \emph{linked}, i.e. non-intersecting, horizontal geodesics. Any pair of linked horizontal geodesics bounds an embedded minimal annulus, the spherical helicoid (see \autoref{bergerhelicoid}), which corresponds to a piece of a vertical unduloid. Therefore we need multiple solution theorems for the existence of tilted or even horizontal unduloids.

\subsection{Vertical and horizontal unduloids in \texorpdfstring{$\mathbb{H}^2\times\R$}{H2xR} and their sisters}\label{section21}The existence of \emph{vertical unduloids}\index{vertical unduloid} in $\mathbb{H}^2\times\R$ as \mcH-surfaces of revolution with $H>\frac{1}{2}$ about the fibre was established by Wu-Teh Hsiang and Wu-Yi Hsiang in \cite{Hsiang}. In \cite{ManzanoTorralbo} it is shown that a spherical helicoid, considered as a surface in $\mathbb{S}^3(4H^2-1,H)$, is the sister surface of such a surface of revolution. We study the spherical helicoid in $\mathbb{S}^3(4H^2-1,H)$ in somewhat more detail, namely we compute the shape operator of it, in order to determine which part of the spherical helicoid corresponds to one half of a vertical unduloid in $\mathbb{H}^2\times\R$. It turns out that for a vertical unduloid the sister surface is bounded by two horizontal circles whose horizontal fields are linearly independent. 

However, for a \emph{horizontal unduloid}, constructed as in \cite{ManzanoTorralbo}, we show that the boundary curves in the vertical mirror plane correspond to horizontal geodesics with the same horizontal field up to a sign. We finish this section with the observation that the sister curves of the horizontal unduloid bound at least two solutions; one solution corresponds to the horizontal unduloid and the other one to a piece of a vertical unduloid.


The first example is the vertical unduloid. It arises as the sister surface of the spherical helicoid\index{spherical helicoid} from \autoref{bergerhelicoid} if we choose $\ell\in\left[0,\frac{2\tau}{\kappa}\right]$. The parameter $c:=\ell/\left(\frac{4\tau}{\kappa}-\ell\right)$ then satisfies $0\leq c\leq1$, allowing us to consider the reparametrisation $f^c$ defined in \autoref{vundushape}. The parameter $c$ describes the \emph{neck size}\index{neck size} of the vertical unduloid, with the extremal cases $c=1$ and $c=0$ corresponding to a vertical cylinder and a sphere, respectively.

\begin{Prop}[for (ii) compare with {{\cite[Proposition 1]{ManzanoTorralbo}}}]\label{vundushape}Let $c\in[0,1]$ and consider the (reparametrised) spherical helicoid
	\[
	f^c\colon\R^2\to\mathbb{S}^3(\kappa,\tau),\qquad f^c(x,y):=\begin{pmatrix}\cos\left(\frac{\sqrt{\kappa}}{2} x\right)\exp\left(-ic\frac{\kappa}{4\tau}y\right)\\\sin\left(\frac{\sqrt{\kappa}}{2} x\right)\exp\left(i\frac{\kappa}{4\tau}y\right)\end{pmatrix}.
	\]
	It is an immersion on $\R^2$ for all $c\in(0,1]$, respectively on $\left(0,\pi/\sqrt{\kappa}\right)\times\R$ for $c=0$. It has the following properties:
	\begin{enumerate}[(i)]
		\item For all $\kappa>0$ and $\tau\in\R$ the spherical helicoid is a minimal surface and the curve $v:=f^c(\pi/\sqrt{\kappa},\cdot)$ is a vertical geodesic on $f^c$. Each meridian $x\mapsto f(x,y)$ is a horizontal $F$-circle with
		\[
		F=\frac{\partial f}{\partial x}(x,y)=\cos\left((c+1)\frac{\kappa}{4\tau}y\right)E_1+\sin\left((c+1)\frac{\kappa}{4\tau}y\right)E_2.\]
		
		\item For $H>\frac{1}{2}$ the sister curve $\widetilde{v}$ of $v$ in $\mathbb{S}^3(4H^2-1,H)$ is a curve of constant curvature $k_{\widetilde{v}}=2H+(c-1)\frac{4H^2-1}{4H}>1$ in a horizontal plane of $\mathbb{H}^2\times\R$. The sister surface of $f^c$ is a surface of revolution with constant mean curvature $H$.\index{vertical unduloid}
		
		\item The curves $h_1:=f^c(\cdot,0)$ and $h_2:=f^c(\cdot,T)$, where
		\[
		T=\frac{\pi}{\sqrt{k_{\tilde{v}}^2-1}}>0,
		\]
		bound the sister surface corresponding to one half of the vertical unduloid $\widetilde{f^c}$. For $c\in(0,1]$ the respective horizontal fields $F_1$ and $F_2$ are linearly independent, that is, they satisfy $F_1\neq
		\pm F_2$. The surface $f^c$ is embedded on $[0,2\pi)\times[0,T]$ for $c\in(0,1]$.
\end{enumerate}
\end{Prop}

\begin{Rem}We note that $f^c$ has also been studied in \cite[Proposition 1]{ManzanoTorralbo}. They sketch the arguments needed to show that the \mcH-sister surface of $f^c$ is rotationally invariant, but they do not determine the piece of $f^c$ corresponding to one half of a vertical unduloid. 
\end{Rem}
\begin{center}
	\begin{figure}[ht]
		\def\svgwidth{0.7\linewidth}
		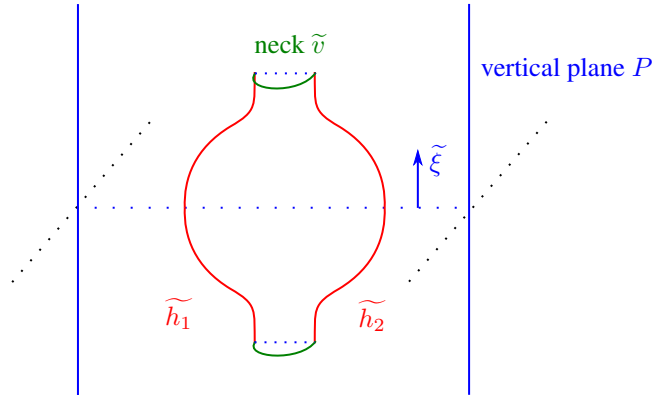
		\caption{Sister surface of $f^c$ in $\mathbb{H}^2\times\R$; $h_1$ and $h_2$ are chosen as in \autoref{vundushape} (iii) and the neck is contained in a horizontal plane with normal $\widetilde{\xi}$}
		\label{parametrizationh2}
	\end{figure}
\end{center} 
The extremal values for $c$ are instructive and useful for the proof of (iii), so we first consider an example before passing to the proof of \autoref{vundushape}.

\begin{Ex}[Vertical cylinder and sphere]\label{exverticalcylinder}We look at $c=1$ and $c=0$.
	\begin{enumerate}[(a)]
		\item The vertical cylinder corresponds to $c=1$. In that case we have $k_{\widetilde{v}}=2H$ and thus $T=\frac{\pi}{\sqrt{4H^2-1}}$. By \autoref{vundushape} (i) and (iii) we have $h_1^\prime=E_1$ and
		\[
		h_2^\prime=\cos\left(\frac{\sqrt{4H^2-1}}{2H}\pi\right)E_1+\sin\left(\frac{\sqrt{4H^2-1}}{2H}\pi\right)E_2.
		\]
		We have $\frac{\sqrt{4H^2-1}}{2H}\in(0,1)$ for $H>\frac{1}{2}$ so that $F_2\neq-F_1$.
		\item The \mcH-sphere corresponds to $c=0$. We then have $k_{\tilde{v}}=2H-\frac{4H^2-1}{4H}=\frac{4H^2+1}{4H}$ and consequently
		\[
		T=\frac{\pi}{\sqrt{\left(\frac{4H^2+1}{4H}\right)^2-1}}=\pi\cdot \frac{4H}{4H^2-1}.\]
		This shows $F_1=-F_2$. In fact, $h_1$ and $h_2$ are part of the same horizontal geodesic, just with opposite orientations.
	\end{enumerate}
\end{Ex}

\begin{proof}[Proof of \autoref{vundushape}]For (i) we argue as in \autoref{bergerhelicoid} to see that $f^c$ is a minimal surface. It is also clear that $v$ is a vertical geodesic since its Hopf projection is a point. The claim about the meridians and the horizontal field $F$ follows in view of \eqref{hgeodesic}.
	
	(ii): Let $v_1:=\frac{\partial f^c}{\partial x}(x,y)$ and $v_2:=\frac{\partial f^c}{\partial y}(x,y)$. At $x=\pi/\sqrt{\kappa}$ we have
	\[
	v_1=\cos\left((c+1)\frac{\kappa}{4\tau}y\right)E_1+\sin\left((c+1)\frac{\kappa}{4\tau}y\right)E_2\quad\mbox{and}\quad v_2=-\xi.
	\]
	Thus $N=-\sin\left((c+1)\frac{\kappa}{4\tau}y\right)E_1+\cos\left((c+1)\frac{\kappa}{4\tau}y\right)E_2$ is the normal at $x=\pi/\sqrt{\kappa}$. Finally we note $\nabla_{v_2}v_1=\left((c-1)\frac{\kappa}{4\tau}+\tau\right)N$ at $x=\pi/\sqrt{\kappa}$, so that the shape operator is
	\[
	S=\begin{pmatrix}0&(c-1)\frac{\kappa}{4\tau}+\tau\\(c-1)\frac{\kappa}{4\tau}+\tau&0\end{pmatrix}.
	\]
	For $\kappa=4H^2-1$ and $\tau=H$ Daniel's correspondence and \eqref{eq:daniel} imply
	\[
	k_{\widetilde{v}}=H+\left\langle S\begin{pmatrix}0\\1\end{pmatrix},\begin{pmatrix}0&1\\-1&0\end{pmatrix}\begin{pmatrix}0\\1\end{pmatrix}\right\rangle=2H+(c-1)\frac{4H^2-1}{4H}.
	\]
	We now show $k_{\widetilde{v}}>1$:
	\begin{align*}
	k_{\widetilde{v}}&=2H+(c-1)\frac{4H^2-1}{4H}=\frac{8H^2+(c-1)(4H^2-1)}{4H}\\
	&=\frac{8H^2+c(4H^2-1)-(4H^2-1)}{4H}=\frac{4H^2+1+c(4H^2-1)}{4H}\\
	&>\frac{4H^2+1}{4H}=\frac{4H^2-4H+1+4H}{4H}=\frac{(2H-1)^2+4H}{4H}>\frac{4H}{4H}=1.
	\end{align*}
	Curves in $\mathbb{H}^2$ with constant geodesic curvature greater than $1$ are rotationally invariant, which proves the claim about the sister surface of $f^c$.
	
	(iii): Finally we compute which part of $f^c$ corresponds to one half of the vertical unduloid $\widetilde{f^c}$. We need the following auxiliary result to have a reference for ``one half'':
	
	\begin{Lem}Let $c\colon[0,L]\to\mathbb{H}^2$ be a unit-speed parametrisation of a simple closed circle in $\mathbb{H}^2$ with constant geodesic curvature $k_{c}>1$. Then  $R=\coth^{-1}(k_{c})$ is the intrinsic radius of $c$.\end{Lem}
	\begin{proof}Let $D$ denote the closed disc bounded by $c$. By Gau\ss{}-Bonnet we have
		\[
		\int_{D}(-1)\,dA_{\mathbb{H}^2}+\int_{c}k_{c}\,dt=2\pi\chi(D).
		\]
		For the (unknown) intrinsic radius $R$ we know $L=2\pi\sinh(R)$ and thus the area of $D$ in $\mathbb{H}^2$ equals $2\pi(\cosh(R)-1)$. Since $\chi(D)=1$ we get
		\[
		-2\pi(\cosh(R)-1)+2\pi\sinh(R)k_c=2\pi\iff k_c=\coth(R).\qedhere
		\]
	\end{proof}
	\emph{Proof of \autoref{vundushape} continued:} We know that $\widetilde{v}$ is a curve of constant geodesic curvature $k_{\widetilde{v}}>1$. Its projection onto $\mathbb{H}^2$ is therefore a curve $c_R$ with intrinsic radius $R=\coth^{-1}(k_{\tilde{v}})$. The circumference of $c_R$ is $2\pi\sinh(R)$, so that one half of a vertical unduloid is realised at $T=\pi\sinh(\coth^{-1}(k_{\tilde{v}}))$. 
	
	Using $\sinh(x)=\frac{\exp(x)-\exp(-x)}{2}$ and $\coth^{-1}(y)=\frac{1}{2}\ln\left(\frac{1+y}{1-y}\right)$ we obtain
	\[
	T=\frac{\pi}{\sqrt{k_{\tilde{v}}^2-1}}.
	\]
	In order to prove the claim about the horizontal fields of $h_1$ and $h_2$ we will show that
	\[\mathcal{T}\colon[0,1]\to\R,\qquad c\mapsto (c+1)\frac{4H^2-1}{4H}\frac{\pi}{\sqrt{k_{\tilde{v}}^2-1}}	
	\]
	is strictly decreasing on $[0,1)$. In \autoref{exverticalcylinder} we computed $\mathcal{T}(0)=\pi$ and $\mathcal{T}(1)\in(0,\pi)$, so that monotonicity of $\mathcal{T}$ on $[0,1)$ shows $\mathcal{T}(c)\in(0,\pi)$ for all $c\in(0,1]$.
	
	For the monotonicity of $\mathcal{T}$ we note 
	\[
	\mathcal{T}(c)=\left(k_{\tilde{v}}-\frac{1}{2H}\right)\frac{\pi}{\sqrt{k_{\tilde{v}}^2-1}}=g(k_{\tilde{v}}(c)),
	\]
	where 
	\[
	g\colon \left[k_{\tilde{v}}(0),k_{\tilde{v}}(1)\right)\to\R,\qquad g(t):=\left(t-\frac{1}{2H}\right)\frac{\pi}{\sqrt{t^2-1}}\]
	is strictly decreasing and $k_{\tilde{v}}$ strictly increasing.

We use \autoref{bergerhelicoid} (ii) to prove that the surface $f^c$ is embedded on $[0,2\pi)\times[0,T]$ for $c\in(0,1]$. In view of this result, it is sufficient to verify $c T<4\tau\pi/\kappa$. In our case we have $\kappa=4H^2-1$ and $\tau=H$, so that this is equivalent to
\[c\frac{4H^2-1}{4H}\frac{\pi}{\sqrt{k_{\tilde{v}}^2-1}}<\pi\mbox{.}\]
For $c=1$ we have $k_{\tilde{v}}=2H$ and the claim is true due to $H>1/2$. The representation $k_{\tilde{v}}=2H+(c-1)(4H^2-1)/(4H)$ implies $c \frac{4H^2-1}{4H}=k_{\tilde{v}}-\frac{4H^2+1}{4H}$. The claim now follows if
\[
 h\colon \left[k_{\tilde{v}}(0),k_{\tilde{v}}(1)\right)\to\R\mbox{,}\qquad h(t):=\left(t-\frac{4H^2+1}{4H}\right)\frac{\pi}{\sqrt{k_{\tilde{v}}^2-1}}
\]
is strictly increasing. A computation and $t>1$ due to $k_{\tilde{v}}>1$ yield
\[
 h^\prime(t)=\pi\frac{\frac{4H^2+1}{4H}t-1}{(t^2-1)^{3/2}}>\pi\frac{4H^2+1-4H}{(t^2-1)^{3/2}}=\pi\frac{(2H-1)^2}{(t^2-1)^{3/2}}>0\mbox{.}\qedhere
\]

\end{proof}

There are more restrictions on the boundary curves of a horizontal unduloid:\index{horizontal unduloid}
\begin{Prop}\label{horizontalunduloid}
	Assume $\widetilde{\Sigma}$ is a singly periodic properly (Alexandrov) embedded \mcH-annulus in $\mathbb{H}^2\times\R$ with $H>\frac{1}{2}$ and horizontal axis $\tilde{\gamma}$. Then the sister curves $h_1$ and $h_2$ from the intersection of $\widetilde{\Sigma}$ with its vertical mirror plane are horizontal geodesics and their respective horizontal fields $F_1$ and $F_2$ satisfy $F_1=-F_2$.
\end{Prop}
\begin{proof}
	The surface $\widetilde{\Sigma}$ has a horizontal mirror plane $P$ by \autoref{singleclass} (ii). By \autoref{extensionhom} the reflection through $P$ corresponds to a geodesic reflection about a vertical geodesic $v$ contained in the minimal sister surface $\Sigma\subset\mathbb{S}^3(4H^2-1,H)$. The surface $\Sigma$ is invariant under this rotation, which implies $F_1=-F_2$.
\end{proof}
Let us consider the converse of \autoref{horizontalunduloid}: Given horizontal geodesics $h_1$ and $h_2$ in $\mathbb{S}^3(4H^2-1,H)$ with horizontal fields $F_1$ and $F_2$ satisfying $F_1=-F_2$. Do they bound a minimal surface whose sister surface is one half of a horizontal unduloid? They do for the case that $h_1$ and $h_2$ can be joined by a segment of a vertical geodesic $v$ and $\lambda=L(v)$ within a certain range. The parameter $\lambda$ corresponds to the neck size of a horizontal unduloid. In fact, they bound also another embedded minimal annulus corresponding to a piece of a vertical unduloid:

\begin{Theorem}\label{examplehorizontal}
	There is a one-parameter family of horizontal geodesics $h_1$ and $h_2$ in $\mathbb{S}^3(4H^2-1,H)$ with horizontal fields $F_1$ and $F_2$ satisfying $F_1=-F_2$ such that each pair bounds two embedded minimal annuli $\Sigmav$ and $\Sigmanv$. The annulus $\Sigmav$ corresponds to a piece of a vertical unduloid in $\mathbb{H}^2\times\R$ and $\Sigmanv$ corresponds to one half of a horizontal unduloid in $\mathbb{H}^2\times\R$. 
\end{Theorem}

\begin{proof}We use the construction of Manzano and Torralbo in \cite{ManzanoTorralbo} in order to establish existence of the embedded minimal annulus $\Sigmanv$. They constructed one quarter of a horizontal unduloid by solving a Plateau problem in the Berger sphere $\mathbb{S}^3(4H^2-1,H)$. The boundary curve consists of segments of three horizontal geodesics and one vertical geodesic; the Hopf projection is a convex sector in $\mathbb{S}^2(4H^2-1)$, so that there is a unique graphical solution for the Plateau problem.
	
	For $\kappa>0$ and $\tau\in\R$ let $\lambda\in\left[0,\frac{\pi}{2\sqrt{\kappa}}\right]$. Moreover let $h_1$ and $h_2$ be the curves
	\begin{align*}
	h_1(t)&=\begin{pmatrix}
	\cos\left(\frac{\sqrt{\kappa}}{2}t\right)\\
	\sin\left(\frac{\sqrt{\kappa}}{2}t\right)
	\end{pmatrix},\\
	h_2(t)&=\begin{pmatrix}
	\cos\left(\sqrt{\kappa}\lambda\right)\cos\left(\frac{\sqrt{\kappa}}{2}t\right)+i\sin\left(\sqrt{\kappa}\lambda\right)\sin\left(\frac{\sqrt{\kappa}}{2}t\right)\\
	-\cos\left(\sqrt{\kappa}\lambda\right)\sin\left(\frac{\sqrt{\kappa}}{2}t\right)+i\sin\left(\sqrt{\kappa}\lambda\right)\cos\left(\frac{\sqrt{\kappa}}{2}t\right)
	\end{pmatrix}.
	\end{align*}
	In view of \eqref{hgeodesic} we notice the following: $h_1$ is a horizontal geodesic through $(1,0)$ with horizontal field $F_1=E_1$ and $h_2$ is a horizontal geodesic through $(\cos(\sqrt{\kappa}\lambda),i\sin(\sqrt{\kappa}\lambda))$ with horizontal field $F_2=-E_1$. We claim that $h_1$ and $h_2$ have the following properties:
	\begin{enumerate}[(a)]
		\item There is a vertical geodesic $v$ joining $h_1\left(\frac{\pi}{2\sqrt{\kappa}}\right)$ and $h_2\left(-\frac{\pi}{2\sqrt{\kappa}}\right)$. The length $\ell$ of $v$ is in the interval $\left[0,\frac{2\tau}{\kappa}\pi\right]$. Thus $h_1$ and $h_2$ are, up to a left translation, a reparametrisation contained in a spherical helicoid $f^c$ for $c\in[0,1]$; compare with \autoref{vundushape}.
		
		\item There is a closed geodesic polygon $\Gamma_\lambda=\gamma_1\oplus\gamma_2\oplus\gamma_3\oplus\gamma_4$ such that $\gamma_1=h_1|_{[0,\pi/\sqrt{\kappa}]}$, $\gamma_2$ and $\gamma_4$ are horizontal geodesics intersecting $\gamma_1$ orthogonally, and $\gamma_3$ is a vertical geodesic joining $\gamma_2$ and $\gamma_4$. The Hopf projection $\Pi\circ\Gamma_\lambda$ is a convex sector in $\mathbb{S}^2(\kappa)$ for $\lambda$ as chosen above. By \cite{ManzanoTorralbo} it bounds a unique minimal graph $\Sigma_\lambda$ (graph with respect to $\Pi$).
		
		\item Geodesic reflection across $\gamma_3$ maps $h_1$ onto $h_2$. Reflections across $\gamma_2$ and $\gamma_4$ extend the surface $\Sigma_\lambda$ to an embedded minimal annulus bounded by $h_1$ and $h_2$.
	\end{enumerate}
	For (a) we exhibit $v$ explicitly. We note first 
	\[
	h_1\left(\frac{\pi}{2\sqrt{\kappa}}\right)=\frac{1}{\sqrt{2}}\begin{pmatrix}1\\1\end{pmatrix}\quad\mbox{ and }\quad h_2\left(-\frac{\pi}{2\sqrt{\kappa}}\right)=\frac{1}{\sqrt{2}}\begin{pmatrix}\exp\left(-i\sqrt{\kappa}\lambda\right)\\
	\exp\left(i\sqrt{\kappa}\lambda\right)                                                                                                                                                       \end{pmatrix}.
	\]
	The curve \[
	v(s)=\frac{1}{\sqrt{2}}\begin{pmatrix}\exp\left(-i\frac{\kappa}{4\tau}s\right)\\
	\exp\left(i\frac{\kappa}{4\tau}s\right)                             
	\end{pmatrix}
	\]
	is a vertical geodesic, see also \eqref{vgeodesic}, that joins $v(0)=\frac{1}{\sqrt{2}}(1,1)$ and
	\[
	v(\ell)=h_2\left(-\frac{\pi}{2\sqrt{\kappa}}\right)\mbox{ for }\ell=\frac{4\tau}{\sqrt{\kappa}}\lambda.
	\]
	By definition of $\lambda$ we have $0\leq\ell\leq \frac{2\tau}{\kappa}\pi$, as claimed.
	
	For (b) we define curves $\gamma_1$ to $\gamma_4$ as follows:
	\begin{align*}
	\gamma_1(t)&=\begin{pmatrix}
	\cos\left(\frac{\sqrt{\kappa}}{2}t\right)\\
	\sin\left(\frac{\sqrt{\kappa}}{2}t\right)
	\end{pmatrix}\mbox{,}&t\in\left[0,\frac{\pi}{\sqrt{\kappa}}\right]\mbox{,}\\
	\gamma_2(t)&=\begin{pmatrix}
	\cos\left(\frac{\sqrt{\kappa}}{2}t\right)\\
	i\sin\left(\frac{\sqrt{\kappa}}{2}t\right)
	\end{pmatrix}\mbox{,}&t\in\left[0,\lambda\right]\mbox{,}\\
	\gamma_3(t)&=\begin{pmatrix}
	\exp\left(i\frac{\kappa}{4\tau}t\right)\cos\left(\frac{\sqrt{\kappa}}{2}\lambda\right)\\
	\exp\left(-i\frac{\kappa}{4\tau}t\right)i\sin\left(\frac{\sqrt{\kappa}}{2}\lambda\right)
	\end{pmatrix}\mbox{,}&t\in\left[0,\frac{2\tau}{\kappa}\pi\right]\mbox{,}\\
	\gamma_4(t)&=\begin{pmatrix}
	i\sin\left(\frac{\sqrt{\kappa}}{2}t\right)\\
	\cos\left(\frac{\sqrt{\kappa}}{2}t\right)
	\end{pmatrix}\mbox{,}&t\in\left[0,\frac{\pi}{\sqrt{\kappa}}-\lambda\right]\mbox{.}
	\end{align*}
	\begin{center}
		\begin{figure}
			\def\svgwidth{0.7\linewidth}
			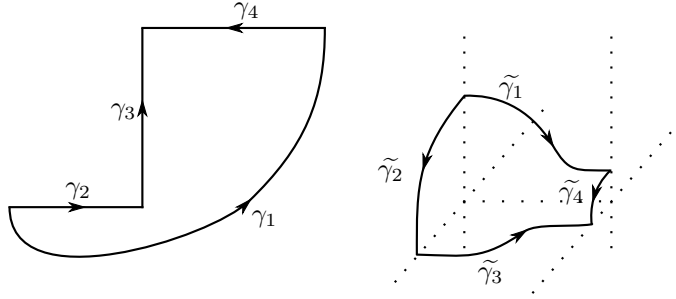
			\caption{Qualitative sketches of $\Gamma_\lambda$ on the left and of the sister contour $\widetilde{\Gamma_\lambda}$ on the right side; compare also with \cite[Figure 3]{ManzanoTorralbo}}
			\label{polygonhunduloid}
		\end{figure}
	\end{center}
	One can check that this defines a closed polygon $\Gamma_\lambda$ oriented as in \autoref{polygonhunduloid}. Looking at \eqref{hgeodesic} and \eqref{vgeodesic} we see the following: $\gamma_1$ is $E_1$-circle, $\gamma_2$ and $\gamma_4$ are $E_2$-circles (in particular they intersect $\gamma_1$ orthogonally) and $\gamma_3$ is a vertical geodesic. By choice of $\lambda$ the sector $\Pi\circ\Gamma_\lambda$ is convex.

	Finally we check (c). Let $\rho_0(z,w):=(z,-w)$ denote the geodesic reflection across the vertical geodesic
	\[
	v_0(t):=\begin{pmatrix}
	\exp\left(i\frac{\kappa}{4\tau}t\right)\\
	0
	\end{pmatrix}
	\]
	through $(1,0)$. A computation shows $\gamma_3(t)=\Lcal_{\gamma_2(\lambda)}(v_0(t))$, so that
	\[
	\rho(z,w):=\left(\Lcal_{\gamma_2(\lambda)}\circ\rho_0\circ\Lcal_{\left(\gamma_2(\lambda)\right)^{-1}}\right)(z,w)
	\]
	is the geodesic reflection across $\gamma_3$. Another computation yields $h_2(t)=\rho(h_1(t))$. Reflections across $\gamma_2$ and $\gamma_4$ extend the surface to the desired embedded minimal annulus $\Sigmanv$ bounded by $h_1$ and $h_2$: This is because the length of the segment $\gamma_1$ is one quarter of the length of $h_1$ or $h_2$, and thus we go once around $h_1$ and $h_2$, respectively. 
	
	The other solution $\Sigmav$ is the spherical helicoid bounded by $h_1$ and $h_2$. By property (a) they are joined by a segment of length $\ell\in\left[0,\frac{2\tau}{\kappa}\pi\right]$. This implies $\Sigmav$ is embedded, see \autoref{bergerhelicoid}. In \autoref{vundushape} we have shown that the sister surface is a piece of the vertical unduloid in $\mathbb{H}^2\times\R$.
	\end{proof}

\begin{Rem}\label{alexandrovhorizontal}
The embeddedness of horizontal unduloids in \cite{ManzanoTorralbo} relies on a generalised Krust theorem by Chuaqui and Hauswirth which has never been published. Therefore we cannot say that horizontal unduloids are embedded. It is known that the horizontal cylinder is embedded, which corresponds to $\lambda=\frac{\pi}{2\sqrt{\kappa}}$ in the proof above. In particular, the cylinder is Alexandrov embedded. Since Alexandrov embeddedness is preserved under continuous deformations, we know that some horizontal unduloids are in fact Alexandrov embedded.

Let us give a simple argument why all horizontal unduloids are Alexandrov embedded. The piece $\Sigma_\lambda$ in the Berger sphere is a minimal graph which is never vertical. The \mcH-sister surface $\widetilde{\Sigma}_\lambda$ in $\mathbb{H}^2\times\R$ is never vertical as well by the Daniel correspondence. The curve $\tilde{\gamma}_3$ is in a horizontal plane of $\mathbb{H}^2\times\R$ above which $\widetilde{\Sigma}_\lambda$ lies due to the maximum principle. The surface $\widetilde{\Sigma}_\lambda$ is on one side of the vertical plane containing $\tilde{\gamma}_1$: This is due to $\widetilde{\Sigma}_\lambda$ being never vertical and because of the maximum principle. Therefore $\widetilde{\Sigma}_\lambda$ bounds an immersed three-ball, that is, it is Alexandrov embedded. The surface $\widetilde{\Sigma}_\lambda$ extends to an Alexandrov embedded annulus by reflection through the vertical and horizontal planes.
\end{Rem}


To conclude this subsection, we complement our picture with unduloids in $\R^3$ and their sisters in $\mathbb{S}^3$:
\begin{Rem}\label{unduloidr3}
	Let $\widetilde{f^n}$ parametrise a vertical unduloid in $\R^3$ with \emph{neck size} $n\in(0,\pi)$. As has been shown in \cite{KarstenProceeding} or \cite{KarstenDiss}, the cousin in $\mathbb{S}^3$ is then parametrised by
	\[f^n(x,y)=\begin{pmatrix}
	\cos(x)\exp(-iny)\\\sin(x)\exp(i(2\pi-n)y)
	\end{pmatrix}.\]
	We see that $f^n$ agrees with the spherical helicoid\index{spherical helicoid} from \autoref{bergerhelicoid} for $\kappa=4$, $\tau=1$ with parameters $\varphi=\pi$ and $\ell=\frac{n}{2}$. For $n=\pi$ we get cylinders in $\R^3$ and for $n=0$ a chain of spheres. 
	

	
	
	We define the family $(\Phi_\theta)_{\theta\in\R}$ by
	\[
	\Phi_\theta\colon\mathbb{S}^3\to\mathbb{S}^3,\qquad\Phi_\theta(z,w):=\begin{pmatrix}\cos\left(\theta/2\right)z-\sin\left(\theta/2\right)w\\                     
	\sin\left(\theta/2\right)z+\cos\left(\theta/2\right)w\end{pmatrix}.
	\]
	This family satisfies $ \frac{d}{d\theta}\Phi_\theta(z,w)=\frac{1}{2}E_1$, i.e., $(\Phi_\theta)_{\theta\in\R}$ is the flow of a Killing vector field on $\mathbb{S}^3$ and thus a one-parameter family of isometries.
	
	The curves $h_1:=f^n(\cdot,0)$ and $h_2:=f^n(\cdot,1/2)$ are horizontal geodesics with horizontal field $\pm E_1$. They bound the embedded minimal annulus $\Sigmav:=f^n(\R\times[0,1/2])$ in $\mathbb{S}^3$. Therefore the boundary of $\Sigmav$ is invariant under the flow $(\Phi_\theta)_{\theta\in\R}$. Since $\Phi_\theta$ is an isometry, each surface $\Sigmav^\theta:=\Phi_\theta(\Sigmav)$ is an embedded minimal annulus bounded by $h_1$ and $h_2$. It can be shown that they correspond to tilted unduloids in $\R^3$, see \cite[Theorem 4.8]{VrzinaPhD} in the author's Ph.D. thesis.\end{Rem}

\subsection{A multiple solution theorem in the Berger spheres}\label{multiplesolution}

The Plateau problem of finding minimal surfaces bounded by two closed curves is interesting in itself. The solution  depends on the topology of the surface we are looking for. We are interested in embedded minimal annuli. For our particular problem, embedded minimal annuli bounded by linked horizontal geodesics in $\mathbb{S}^3(\kappa,\tau)$, two questions arise: Is there a minimal annulus bounded by these two curves? How many solutions exist?

We have already answered the first question: in the Berger spheres we have exhibited one explicit solution in \autoref{bergerhelicoid}, corresponding to vertical unduloids in $\mathbb{H}^2\times\R$. We have shown there is a unique embedded spherical helicoid for a given orientation of the boundary curves. Regarding the second question, we have seen in \autoref{examplehorizontal} that certain horizontal geodesics $h_1$ and $h_2$ with horizontal fields satisfying $F_1=-F_2$ bound at least two embedded minimal annuli.

For a general pair of linked horizontal geodesics $h_1$ and $h_2$ in the Berger spheres the question regarding the number of minimal annuli bounded by $h_1$ and $h_2$ is more delicate. In a series of papers, see \cite{Min89,Min92,Min93}, Min Ji proves a minimax principle to obtain a multiple solution theorem for a pair of linked curves in $\mathbb{S}^3$, or more generally in any compact Riemannian three-manifold. This minimax principle is applicable to a pair of linked horizontal geodesics in a Berger sphere: Given horizontal geodesics $h_1$ and $h_2$ in $\mathbb{S}^3(\kappa,\tau)$ with a prescribed orientation, there are at least two (possibly branched) minimal annuli bounded by $h_1$ and $h_2$. We do not include the details in the present paper since we need a stronger result. For instance, embeddedness of the solution is not clear. A more detailed account of this multiple solution theorem can be found in \cite[Chapter 5 and Appendix B]{VrzinaPhD}. 

\section{Existence of tilted unduloids follows from a uniqueness problem}\label{section3}
We have narrowed our conjugate Plateau construction of tilted unduloids as follows:
\begin{itemize}
\item Choose horizontal geodesics $h_1$ and $h_2$ in $\mathbb{S}^3(4H^2-1,H)$ with linearly independent horizontal fields and fix their orientations. Construct an embedded minimal annulus $\Sigma_0$ bounded by the oriented circles $h_1$ and $h_2$ (different from the spherical helicoid), and consider the universal cover $\Sigma$ of $\Sigma_0$.
\item Daniel's correspondence yields an \mcH-surface $\widetilde{\Sigma}$ in $\mathbb{H}^2\times\R$.
\item Extend $\widetilde{\Sigma}$ to a tilted unduloid by reflection.
\end{itemize}
In \autoref{section2} we observed that neither the boundary curves nor an abstract Plateau solution carry sufficient information for the sister surface to be a tilted unduloid. We show that this construction works under an additional assumption concerning the number of embedded minimal annuli bounded by horizontal geodesics (for a fixed orientation) in the Berger spheres. Namely, we need to assume there are exactly two such annuli: we are as yet unable to verify this assumption. In any case, the existence problem for tilted unduloids is reduced to a uniqueness problem in the Berger spheres.

\subsection{Conjugate construction of tilted unduloids in \texorpdfstring{$\mathbb{H}^2\times\R$}{H2xR} under hypotheses}\label{chapter61}
It is useful to introduce the following notation for the conjugate Plateau construction\index{conjugate Plateau construction} outlined in the introduction to this section:
\begin{Notation}
	Let $h_1$ and $h_2$ be linked horizontal geodesics in $\mathbb{S}^3(\kappa,\tau)$ with prescribed orientations. We set
	\begin{align*}
	\Mcal_0(h_1,h_2)&:=\left\{\Sigma_0\colon\Sigma_0\mbox{ is an embedded minimal annulus bounded by }h_1\mbox{ and }h_2\right\}\\
	\intertext{and}
	\Mcal(h_1,h_2)&:=\left\{\Sigma\colon\Sigma\mbox{ is the minimal universal cover of }\Sigma_0\in\Mcal_0\right\}.
	\end{align*}
	We also write $\Mcal_0=\Mcal_0(h_1,h_2)$ and $\Mcal=\Mcal(h_1,h_2)$.
\end{Notation}
We finally pose the existence question: Do tilted unduloids in $\mathbb{H}^2\times\R$ exist? An affirmative answer to our problem depends on the following hypotheses whose geometric meaning is explained in \autoref{tiltedsolution}:
\begin{Def}[Hypotheses]
	Let $H>\frac{1}{2}$ and let $h_1$ and $h_2$ be linked horizontal geodesics with prescribed orientations in $\mathbb{S}^3(4H^2-1,H)$, having length $L$. Let $\Sigma\in\Mcal(h_1,h_2)$ and consider the \mcH-sister surface $\widetilde{\Sigma}$ in $\mathbb{H}^2\times\R$. By $P_j$ we denote the vertical plane containing the sister curve $\widetilde{h_j}$, where $j\in\{1,2\}$. The minimal surface $\Sigma$ \emph{satisfies (H1) to (H3)} if the following is true:
	\begin{enumerate}[({H}1)]
		\item Either there exists $j\in\{1,2\}$ such that $\widetilde{h_j}(0)\neq\widetilde{h_j}(L)$ or if $\widetilde{h_j}(0)=\widetilde{h_j}(L)$ for all $j\in\{1,2\}$ then $P_1\cap P_2=\emptyset$.
		\item The \mcH-surface $\widetilde{\Sigma}$ has a non-vertical axis.
		\item If $\widetilde{\Sigma}$ extends to an \mcH-annulus by reflections through $P_1$ and $P_2$ then the annulus is (Alexandrov) embedded.
	\end{enumerate}
\end{Def}
We state the main technical result:
\begin{Lem}\label{tiltedsolution}
	Let $H>\frac{1}{2}$ and consider linked horizontal geodesics $h_1$ and $h_2$ with horizontal fields $F_1$ and $F_2$ in $\mathbb{S}^3(4H^2-1,H)$. Let $\Sigma\in\Mcal(h_1,h_2)$.
	\begin{enumerate}[(i)]\item Assume the \mcH-sister surface $\widetilde{\Sigma}$ satisfies (H1). Then $\widetilde{\Sigma}$ is a singly periodic \mcH-surface in $\mathbb{H}^2\times\R$.
		\item If $F_1$ and $F_2$ are linearly independent and $\widetilde{\Sigma}$ satisfies (H1) to (H3) then $\widetilde{\Sigma}$ extends to a tilted unduloid in $\mathbb{H}^2\times\R$.
	\end{enumerate}
\end{Lem}

\begin{proof}We prove (i). Let $L:=4\pi/\sqrt{4H^2-1}$ be the length of a horizontal geodesic in $\mathbb{S}^3(4H^2-1,H)$. Then we have $h_1(0)=h_1(L)$ and $h_2(0)=h_2(L)$. When applying the conjugate sister relation, the cases 
	\begin{equation}\mbox{(A)}\quad\widetilde{h_1}(0)\neq \widetilde{h_1}(L)\qquad\mbox{ or }\qquad\mbox{(B)}\quad\widetilde{h_1}(0)=\widetilde{h_1}(L).\label{caseab}\end{equation}
	can occur.
	
	Let us assume (A) first. Then $p:=\widetilde{h_1}(0)$ and $q:=\widetilde{h_1}(L)$ are distinct points in the same vertical plane $P_1$. There is an isometric translation $\Phi$, acting without fix points, such that $\Phi(p)=q$. We claim that the surface $\widetilde{\Sigma}$ is invariant under $\Phi$ and thus singly periodic. In order to show this we want to use the Fundamental Theorem of Surfaces in homogeneous three-manifolds, see \cite[Sections 3 and 4]{Daniel}. 
	
	We consider the annulus $\Sigma_0$ as a minimal immersion $f_0\colon\Omega_0\to\mathbb{S}^3(4H^2-1,H)$ defined on an annulus $\Omega_0\subset\R^2$. The universal cover $\Sigma$ of $\Sigma_0$ is then a minimal immersion $f\colon\Omega\to\mathbb{S}^3(4H^2-1,H)$ defined of the universal cover $\Omega$ of $\Omega_0$. The geometric data are the first fundamental form, the shape operator $S$, the vertical component of the Gau\ss{} map $\langle Z,N\rangle$, and the vector field $df^{-1}(T)$ as in Proposition 3.8. The \mcH-sister $\tilde{f}$ has the same first fundamental form (Daniel correspondence is isometric) and shape operator $\widetilde{S}=JS+H\id$. Furthermore we know $\langle \widetilde{Z},\widetilde{N}\rangle=\langle Z,N\rangle$ and $df^{-1}(T)=Jd\tilde{f}^{-1}(\widetilde{T})$. Since $\Sigma_0$ is an annulus we see that the geometric data at the points $p$ and $q$ are equal. That is, periodic data of the embedded minimal annulus $\Sigma_0$ in the Berger spheres imply periodic data for the \mcH-sister surface $\widetilde{\Sigma}$. Integrating along curves $\tilde{c}=\tilde{f}\circ\gamma$ joining $p$ and $q$, the Fundamental Theorem of Surfaces in homogeneous three-manifolds implies that 
	$\Phi$ generates the isometry group of the \mcH-surface $\widetilde{\Sigma}$. In particular, $\widetilde{h_2}$ cannot be a closed curve.
	
	Now we consider case (B). If $\widetilde{h_1}$ is closed then $\widetilde{h_2}$ is closed, too. Otherwise we can argue as in case (A) to show that the surface is singly periodic. We assume (H1), that is, the vertical planes $P_1$ and $P_2$ containing $\widetilde{h_1}$ and $\widetilde{h_2}$, respectively, are disjoint. Then we can reflect through $P_1$ and $P_2$ to obtain a singly periodic \mcH-annulus with a horizontal axis. 
	
	Now we prove (ii). We claim that (H2) and (H3) imply case (A) in the proof of (i). Assume that we were in case (B). Then we can extend $\widetilde{\Sigma}$ by reflections through $P_1$ and $P_2$ to an \mcH-annulus $\widetilde{\Sigma}$ with a horizontal axis. Hypothesis (H3) guarantees that $\widetilde{\Sigma}$ is (Alexandrov) embedded, hence it has a vertical mirror plane $P$ by \autoref{singleclass} (ii). Let $\widetilde{h}$ be a mirror curve in $P$, see also \autoref{annuluscaseb}.
	
	\begin{figure}[h]
		\def\svgwidth{0.55\linewidth}
		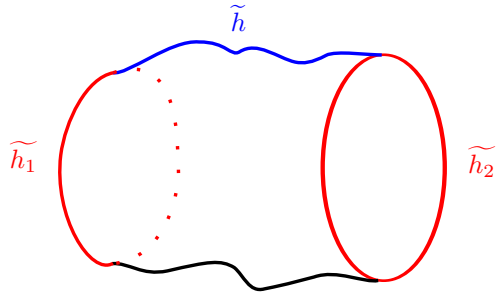
		\caption{Case (B) from \eqref{caseab} is impossible in (ii)}
		\label{annuluscaseb}
	\end{figure}
		
	Then the sister curve $h$ is a horizontal geodesic that intersects $h_1$ and $h_2$ orthogonally. This implies $F_1=\pm F_2$ for the horizontal fields $F_1$ and $F_2$, in contradiction to our assumption that these fields are linearly independent. 
	
	Thus $\widetilde{\Sigma}$ satisfies  case (A) and is singly periodic. By (H2) the axis is non-vertical, so that the vertical planes $P_1$ and $P_2$ must be equal by \autoref{singleprop} (i). Hence reflection $\sigma$ through $P_1=P_2$ extends $\widetilde{\Sigma}$ to a properly immersed \mcH-annulus $\widetilde{\Sigma}\cup\sigma(\widetilde{\Sigma})$. The singly periodic annulus $\widetilde{\Sigma}\cup\sigma(\widetilde{\Sigma})$ is (Alexandrov) embedded by (H3). The axis cannot be horizontal since \autoref{horizontalunduloid} implies $F_1=\pm F_2$, contradicting our assumption on $F_1$ and $F_2$. So indeed the axis is tilted.\end{proof}

\subsection{Discussion of hypotheses in construction and conjecture on tilted unduloids}\label{chapter62}In \autoref{tiltedsolution} we have shown that tilted unduloids in $\mathbb{H}^2\times\R$ exist if there are horizontal geodesics $h_1$ and $h_2$ with linearly independent horizontal fields $F_1$ and $F_2$ such that some $\Sigma\in\Mcal(h_1,h_2)$ satisfies (H1) to (H3). These hypotheses are implied by the following conjecture:

\begin{Conj}[Uniqueness]\label{uniquenessannulus}
	Let $h_1$ and $h_2$ be linked horizontal geodesics in a Berger sphere $\mathbb{S}^3(\kappa,\tau)$ with $\kappa\neq4\tau^2$. Then we have $\lvert\Mcal_0(h_1,h_2)\rvert=2$, that is, there are exactly two embedded minimal annuli bounded by $h_1$ and $h_2$.
\end{Conj}

Why do we believe this conjecture to be valid? Let us first look at the lower dimensional problem of finding embedded geodesic arcs between two points $p$ and $q$ on the standard two-sphere $\mathbb{S}^2$: If $q=-p$ (antipodal or symmetrical case), then there are infinitely many embedded geodesic arcs (induced by a one-parameter family of isometries of $\mathbb{S}^2$), and if $q\neq -p$ (unsymmetrical case) we have exactly two geodesic arcs, ``min'' and ``minimax''.

In a Berger sphere two linked horizontal geodesics $h_1$ and $h_2$ always bound an embedded spherical helicoid. For a given orientation of the boundary curves, the embedded spherical helicoid is unique. Compare also with \autoref{bergerhelicoid}. For our problem, we recall \autoref{unduloidr3}: In $\mathbb{S}^3$, that is, for $\kappa=4\tau^2$ there is an $\mathbb{S}^1$-orbit of embedded minimal annuli bounded by horizontal geodesics $h_1$ and $h_2$ with horizontal fields $E_1$ and $-E_1$. In fact, this critical orbit for the Dirichlet energy is induced by a one-parameter family of isometries in $\mathbb{S}^3$ (symmetrical case). Deforming the metric from the symmetric standard three-sphere to a less symmetric Berger sphere, these do not remain isometries. A special case of this orbit arises when $h_1$ and $h_2$ are contained in a Clifford torus, i.e., the sister surface is one half of a vertical cylinder in $\R^3$ via the Lawson correspondence. One can verify by computation that exactly two surfaces in this orbit remain critical in $\mathbb{S}^3(\kappa,\tau)$, one of them corresponding to a piece of a vertical cylinder in $\mathbb{H}^2\times\R$ and the other surface to one half of a horizontal cylinder in $\mathbb{H}^2\times\R$. The solution corresponding to the vertical cylinder is ``min'' and the other one ``minmax''. Moreover, certain horizontal geodesics $h_1$ and $h_2$ with horizontal fields $E_1$ and $-E_1$ always bound at least two embedded minimal annuli by \autoref{examplehorizontal}.

For the general case we have at least two critical points for the Dirichlet energy, a ``min'' and a ``minimax'', according to a Minimax Principle due to Min Ji. Nevertheless, these techniques do not seem sufficient to verify the uniqueness conjecture; see also \autoref{multiplesolution}.


A similar uniqueness result is available in $\R^3$: In \cite[Theorem 1.2]{MeeksWhite} Meeks and White showed that an extremal pair of smooth disjoint convex curves in distinct planes bounds at most two embedded minimal annuli. The case of two minimal annuli is realised by two generic solutions, a ``min'' and a ``minimax'', even though the case of no solutions occurs. 

While the geometric situation is different, we believe in the basic idea of proof to carry over. We suggest the following steps to prove the conjecture in $\mathbb{S}^3(\kappa,\tau)$:
\begin{enumerate}[(a)]
	\item \emph{Shiffman type theorem:} Show that horizontal geodesics on a Clifford torus with linearly dependent horizontal fields bound exactly two embedded minimal annuli in $\mathbb{S}^3(\kappa,\tau)$.
	\item \emph{Meeks-White degree argument:} Use a degree argument as in \cite{MeeksWhite} to prove that a deformation of the curves considered in (a) bounds only two embedded minimal annuli in $\mathbb{S}^3(\kappa,\tau)$.
\end{enumerate}


If we assume the uniqueness conjecture to be true, we can prove existence of tilted unduloids as a perturbation of horizontal unduloids:
\begin{Theorem}[Perturbation result]\label{perturbationresult}Assume \autoref{uniquenessannulus} to be true. For $c\in(0,1]$ let $f^c$ be the spherical helicoid from \autoref{vundushape} and consider $h_1:=f^c(\cdot,0)$ as well as $h_2^\alpha:=f^c(\cdot,\alpha)$ for 
	\[\alpha\in\left(0,8H\frac{\pi}{(c+1)(4H^2-1)}\right)=:I.\]
	For fixed $c$, we then have two one-parameter families of embedded minimal annuli, namely, $(\Sigmav^\alpha)_{\alpha\in I}$ corresponding to pieces of vertical unduloids, and another family $(\Sigmanv^\alpha)_{\alpha\in I}$.
	Let \[\alpha_0=4H\frac{\pi}{(c+1)(4H^2-1)}.\] 
	Then there exists $\varepsilon\in(0,\alpha_0)$ such that for each $\alpha\in(\alpha_0-\varepsilon,\alpha_0+\varepsilon)\setminus\{\alpha_0\}$ the surface $\widetilde{\Sigmanv^\alpha}$ extends to a tilted unduloid in $\mathbb{H}^2\times\R$.
\end{Theorem}
\begin{proof}
	
	The assumption $\lvert\Mcal_0(h_1,h_2^\alpha)\rvert=2$ and \autoref{examplehorizontal} show that the \mcH-sister surface $\widetilde{\Sigmanv^{\alpha_0}}$ is one half of a horizontal unduloid in $\mathbb{H}^2\times\R$.
	
	Since $\lvert\Mcal_0(h_1,h_2^\alpha)\rvert=2$ for all $\alpha\in I$ we see that $\Sigmanv^\alpha$ depends continuously on $\alpha$. The surface $\Sigmanv^{\alpha_0}$ satisfies (H1) to (H3): Indeed, let $N^\alpha$ and $\widetilde{N}^\alpha$ denote the normals to $\Sigma_2^\alpha$ and $\widetilde{\Sigma_2^\alpha}$, respectively. Then we have
	\begin{equation}
	\langle N^{\alpha_0}\circ h_1,\xi\circ h_1\rangle=\left\langle \widetilde{N}^{\alpha_0}\circ{\widetilde{h_1}},\tilde{\xi}\circ{\widetilde{h_1}}\right\rangle\neq0\label{verticalnormal}
	\end{equation}
	since $c>0$ and because the upper half of the horizontal unduloid is never vertical by the construction of Manzano and Torralbo; see \cite{ManzanoTorralbo} or our \autoref{examplehorizontal} and note that $\widetilde{\gamma_1}$ in \autoref{polygonhunduloid} extends to $\widetilde{h_1}$.
	The continuity property shows that for $\alpha$ close to $\alpha_0$ the surface $\Sigmanv^\alpha$ has no branch points, i.e., it is an immersion, and \eqref{verticalnormal} is preserved for $\alpha$ close to $\alpha_0$. This verifies three properties:
	\begin{itemize}
		\item $\widetilde{h_2}^\alpha(0)\neq\widetilde{h_2}^\alpha(L)$: otherwise the normal along the sister curve $\widetilde{h_2}^\alpha$ in the vertical plane $P_2^\alpha$ would be horizontal, a contradiction to the choice of $\alpha$. Therefore $\widetilde{\Sigmanv^\alpha}$ is singly periodic as shown in \autoref{tiltedsolution} (i).
		
		\item The horizontal fields $F_1$ and $F_2^\alpha$ are linearly independent for all $\alpha\in I\setminus\{\alpha_0\}$ since $F_2^{\alpha_0}=-F_1$ by definition of $\alpha_0$. We refer to \autoref{vundushape} (i) for an explicit computation of the horizontal field.
		
		\item The axis is non-vertical, i.e., (H2) is satisfied: If $\widetilde{\Sigmanv^\alpha}$ had a vertical axis then the normal along $\widetilde{h_1}$ would be horizontal at some point, contradicting the choice of $\alpha$ once again. 
	\end{itemize}
	Recall that horizontal unduloids are Alexandrov embedded due to \autoref{alexandrovhorizontal}. Hypothesis (H3) follows since Alexandrov embeddedness is preserved under continuous deformations.  Such a deformation argument has been used to show that minimal spheres in a compact homogeneous three-manifold are  Alexandrov embedded, see \cite[Lemma 4.3 and Corollary 4.4]{MP}.
	
	This shows the existence of $\varepsilon\in(0,\alpha_0)$ with the claimed properties.
\end{proof}

We believe the conclusion in the perturbation result to be true for all $\alpha\in I\setminus\{\alpha_0\}$ and formulate the conjecture on tilted unduloids as follows:
\begin{Conj}[Tilted unduloid]\label{tiltedconjecture}Let $H>\frac{1}{2}$ and $c\in(0,1]$. We consider the spherical helicoid $f^c$ as in \autoref{vundushape}. Let $h_1:=f^c(\cdot,0)$ and $h_2^\alpha:=f^c(\cdot,\alpha)$ for $\alpha\in I$. Then for each $\alpha\in I\setminus\{\alpha_0\}$ there is $\Sigma^\alpha \in\Mcal(h_1,h_2^\alpha)$ so that $\widetilde{\Sigma^\alpha}$ extends to a tilted unduloid in $\mathbb{H}^2\times\R$.
\end{Conj}
The geometry is not completely analysed. Namely, we cannot determine the exact slope of the axis or the neck size of a tilted unduloid. 

We recall that for the family of spherical helicoids $f^c$ the parameter $c\in(0,1]$ corresponds to the neck size of the vertical unduloid in $\mathbb{H}^2\times\R$. If we do not fix $c$ in the construction above, we obtain two two-parameter families of universal covers of embedded minimal annuli
\[
\left(\Sigmav^{\alpha,c}\right)_{(\alpha,c)\in I\times(0,1)}\quad\mbox{ and }\quad\left(\Sigmanv^{\alpha,c}\right)_{(\alpha,c)\in I\times(0,1)}.
\]
In the first family $\alpha$ parametrises what piece of a vertical unduloid we consider and $c$ is the neck size of the vertical unduloid. For the second family the following natural question arises, provided that the conjectures on the uniqueness of embedded minimal annuli and tilted unduloids are true:
\begin{Question}
	Does $\alpha$ parametrise the slope of the axis of $\widetilde{\Sigmanv^{\alpha,c}}$? Does $c$ correspond to the neck size in this family?
\end{Question}

\bibliography{tiltedunduloid}

\providecommand{\bysame}{\leavevmode\hbox to3em{\hrulefill}\thinspace}
\providecommand{\MR}{\relax\ifhmode\unskip\space\fi MR }
\providecommand{\MRhref}[2]{%
  \href{http://www.ams.org/mathscinet-getitem?mr=#1}{#2}
}
\providecommand{\href}[2]{#2}
\begin{thebibliography}{KKS89}

\bibitem[Dan07]{Daniel}
Beno\^{\i}t Daniel, \emph{Isometric immersions into $3$-dimensional homogeneous
  manifolds}, Commentarii Mathematici Helvetici \textbf{82} (2007), no.~1,
  87--131.

\bibitem[GB93]{KarstenDiss}
Karsten Gro\ss{}e-Brauckmann, \emph{New surfaces of constant mean curvature},
  Mathematische Zeitschrift \textbf{214} (1993), no.~1, 527--565.

\bibitem[GB05]{KarstenProceeding}
\bysame, \emph{{C}ousins of {C}onstant {M}ean {C}urvature {S}urfaces}, Global
  {T}heory of {M}inimal {S}urfaces, vol.~2, Clay Mathematisc Proceedings, 2005,
  pp.~747--767.

\bibitem[HH89]{Hsiang}
Wu-Teh Hsiang and Wu-Yi Hsiang, \emph{On the uniqueness of isoperimetric
  solutions and imbedded soap bubbles in non-compact symmetric spaces}, Invent.
  math. \textbf{98} (1989), 39--58.

\bibitem[Ji89]{Min89}
Min Ji, \emph{An a {P}riori {E}stimate for {D}ouglas {P}roblem in {R}iemannian
  {M}anifolds}, Acta Mathematica Sinica, English Series \textbf{5} (1989),
  no.~3, 235--249.

\bibitem[Ji93a]{Min93}
\bysame, \emph{Minimal {A}nnuli in {R}iemannian {M}anifolds}, Acta Mathematica
  Sinica, English Series \textbf{9} (1993), no.~1, 74--89.

\bibitem[Ji93b]{Min92}
\bysame, \emph{{M}ultiple {S}olutions to the {D}ouglas {P}roblem in
  {$\mathbb{S}^n$}}, SCIENCE CHINA Mathematics \textbf{36} (1993), no.~10,
  1162--1168.

\bibitem[KKS89]{KKS}
Nicholas~J. Korevaar, Robert Kusner, and Bruce Solomon, \emph{The structure of
  complete embedded surfaces with constant mean curvature}, J. Differential
  Geometry \textbf{30} (1989), 465--503.

\bibitem[Maz15]{Mazet}
Laurent Mazet, \emph{Cylindrically bounded constant mean curvature surfaces in
  {$\mathbb{H}^2\times\mathbb{R}$}}, Transactions of the American Mathematical
  Society \textbf{367} (2015), no.~8, 5329--5354.

\bibitem[MP12]{MP}
William~H. Meeks and Joaquin Perez, \emph{Constant {M}ean {C}urvature
  {S}urfaces in metric lie groups}, Contemporay Mathematics \textbf{570}
  (2012), 25--110.

\bibitem[MT14]{ManzanoTorralbo}
Jos\'e~M. Manzano and Francisco Torralbo, \emph{New {E}xamples of {C}onstant
  {M}ean {C}urvature in {$\mathbb{S}^2\times\mathbb{R}$} and
  {$\mathbb{H}^2\times\mathbb{R}$}}, Michigan Math J. \textbf{63} (2014),
  no.~4, 701--723.

\bibitem[MW93]{MeeksWhite}
William~H. Meeks and Brian White, \emph{The {S}pace of {M}inimal {A}nnuli
  {B}ounded by an {E}xtremal {P}air of {P}lanar {C}urves}, Communications in
  Analysis and Geometry \textbf{1} (1993), no.~3--4, 415--437.

\bibitem[Onn08]{Onnis}
Irene~I. Onnis, \emph{Invariant surfaces with constant mean curvature in
  $\mathbb{H}^2\times\mathbb{R}$}, Annali di Matematica \textbf{187} (2008),
  667--682.

\bibitem[Tor12]{Torralbo2}
Francisco Torralbo, \emph{Compact minimal surfaces in the {B}erger spheres},
  Annals of Global Analysis and Geometry \textbf{41} (2012), no.~4, 391--405.

\bibitem[Vrz16]{VrzinaPhD}
Miroslav Vrzina, \emph{Constant {M}ean {C}urvature {A}nnuli in {H}omogeneous
  {M}anifolds}, Ph.D. thesis, Technische Universit\"at Darmstadt, Germany,
  2016.

\end{thebibliography}
\bibliographystyle{amsalpha}
\end{document}